\documentclass[11pt]{article}
\usepackage[font=small,format=plain,labelfont=bf,up]{caption}
\usepackage[a4paper,nohead,ignorefoot,
twoside,scale=0.85,bindingoffset=1.25cm]{geometry}
\usepackage{amssymb,amsfonts,amsmath,amsthm}
\usepackage[]{graphicx}
\usepackage{color}
\graphicspath{{}}
\newcommand{\url}[1]{#1} 
\usepackage{hyperref}%
\definecolor{gray}{rgb}{0.2,0.2,.2}
\hypersetup{%
  unicode=true,          
  colorlinks=true,       
  linkcolor=gray,          
  citecolor=gray      
}
\newcommand{\BMHC}{}
\newcommand{\EMHC}{}
\newcommand{\ul}[1]{\underline{#1}}
 
\newcommand{\ol}[1]{{\overline{#1}}}

\newcommand{\bigpar}{\par\quad\newline\noindent}
\newcommand{\dint}[1]{\,\mathrm{d}#1}

\newcommand{\fspace}[1]{{\mathsf{#1}}}
\newcommand{\fspaceL}{\fspace{L}}

\newcommand{\Rset}{{\mathbb{R}}}

\newcommand{\Nset}{{\mathbb{N}}}

\newcommand{\ccinterval}[2]{[#1,\,#2]}%

\newcommand{\DO}[1]{{O\at{#1}}}

\newcommand{\skp}[2]{{\left\langle{#1},\,{#2}\right\rangle}}

\newcommand{\bskp}[2]{{\big\langle{#1},\,{#2}\big\rangle}}

\newcommand{\at}[1]{{\left({#1}\right)}}

\newcommand{\bat}[1]{{\big(#1\big)}}
\newcommand{\Bat}[1]{{\Big(#1\Big)}}

\newcommand{\quadruple}[4]{{\left({#1},\,{#2},\,{#3},\,{#4}\right)}}

\newcommand{\norm}[1]{\|{#1}\|}
\newcommand{\bnorm}[1]{\big\|{#1}\big\|}

\newcommand{\abs}[1]{\left|{#1}\right|}
\newcommand{\babs}[1]{\big|{#1}\big|}

\newcommand{\ga}{{\gamma}}

\newcommand{\la}{{\lambda}}
\newcommand{\si}{{\sigma}}

\newcommand{\calA}{\mathcal{A}}

\newcommand{\calC}{\mathcal{C}}

\newcommand{\calF}{\mathcal{F}}
\newcommand{\calG}{\mathcal{G}}

\newcommand{\calP}{\mathcal{P}}

\newcommand{\calR}{\mathcal{R}}

\newcommand{\calV}{\mathcal{V}}

\theoremstyle{plain}
\newtheorem{theorem}{Theorem}[]

\newtheorem{proposition}   [theorem]{Proposition}

\newtheorem*{result*}{Main result}
\newtheorem*{problem*}{Open problems}

\newtheorem{assumption} [theorem]{Assumption}
\usepackage{float}
\begin{document}
%
%
%
\title{High-energy waves in superpolynomial FPU-type chains}
\date{\today}
\author{%
  Michael Herrmann\footnote{
     {\tt{michael.herrmann@uni-muenster.de}},
     Institut f\"ur Numerische und Angewandte Mathematik,
     Universit\"at M\"unster
  }
}%
\maketitle
%
%
%
%
\begin{abstract}
We consider periodic traveling waves in FPU-type chains with superpolynomial interaction forces and derive explicit asymptotic formulas for the high-energy limit as well as bounds for the corresponding approximation error. In the proof we adapt twoscale techniques that have recently been developed by Herrmann and Matthies for chains with singular potential and provide an asymptotic ODE for the scaled distance profile.
\end{abstract}
%
%
%
\quad\newline\noindent%
\begin{minipage}[t]{0.15\textwidth}%
  Keywords:
\end{minipage}%
\begin{minipage}[t]{0.8\textwidth}%
\emph{asymptotic analysis}, \emph{nonlinear lattice waves}, \emph{high-energy limit in atomic chains}
\end{minipage}%
\medskip
\newline\noindent
\begin{minipage}[t]{0.15\textwidth}%
  MSC (2010):
\end{minipage}%
\begin{minipage}[t]{0.8\textwidth}%
37K60,  
37K40,  
74H10  	
\end{minipage}%
%
%
%
\setcounter{tocdepth}{4}
\setcounter{secnumdepth}{2}
{\scriptsize{\tableofcontents}}
%

%
%
\section{Introduction}\label{sec:Intro}
%

Hamiltonian \BMHC lattices \EMHC appear in many branches of 
physics and materials science as simple dynamical models for spatially discrete systems with energy conservation. Of particular importance are coherent structures such as traveling waves, which can be regarded as the nonlinear fundamental modes and provide the building blocks for more complex solutions. 
\par
From a mathematical point of view, Hamiltonian lattice waves are rather complicate because they are not determined by ODEs -- as the counterparts in spatially continuous systems -- but by nonlocal advance-delay-differential equations. In the prototypical case of a one-dimensional chain with nearest neighbor interactions, which we refer to as Fermi-Pasta-Ulam or FPU-type chain, the problem consists in finding scalar profile functions $\calV$ and $\calR$ along with a speed $\si$ such that
\begin{align}
\label{Eqn:TW}
\BMHC-\EMHC\si \calR^\prime\at{x}=\calV\at{x+\tfrac12}-\calV\at{x-\tfrac12}\,,\qquad 
\BMHC-\EMHC\si \calV^\prime\at{x}=\Phi^\prime\Bat{\calR\at{x+\tfrac12}}-\Phi^\prime\bat{\calR\at{x-\tfrac12}}
\end{align}
is satisfied for all $x\in\Rset$. Here, $\Phi^\prime$ is the derivative of
the interaction potential and $x=j-\si t$ denotes the continuous phase variable, that is
the spatial coordinate in the co-moving frame. The profiles $\calV$ and $\calR$ are related to the atomic velocities and distances, respectively, since we have
\begin{align*}
\dot{u}_j\at{t}=\calV\at{j-\si t}\,,\qquad u_{j+1}\at{t}-u_j\at{t}=\calR\bat{j+\tfrac12-\si t}
\end{align*}
for any traveling wave in  the FPU-type chain $\ddot{u}_j=\Phi^\prime\at{u_{j+1}-u_j}-\Phi^\prime\at{u_j-u_{j-1}}$.
\par
For the integrable examples -- the linear case, the Toda chain, and  the hard-sphere model -- the complete solution set to the nonlocal equation \eqref{Eqn:TW} is known, see \cite{Tes01,DHM06} and the references therein. For general potentials, however, there exists nowadays a sophisticated existence theory for different types of traveling waves, see the discussion below, but only very little is known about their qualitative properties. In particular, there is no result that guarantees the uniqueness or the dynamical stability of traveling waves in a chain with arbitrary but convex interaction potential. Mathematically, this lack of information is intimately related to the characteristic features of \eqref{Eqn:TW}.
First, the advance-delay structure prevents the exploitation of many concepts and methods that have been proven powerful in the context of low-dimensional ODEs and delay equations. Second, the underlying lattice flow is symplectic and exhibits a richer dynamical behavior than dissipative particle systems. 
\par
Since a complete understanding of the general case is currently out of reach, we are  
naturally interested in special potentials and asymptotic regimes that allow for almost explicit or approximate solutions of \eqref{Eqn:TW} and hence to investigate the open mathematical problems in a simplified setting. The most prominent example is the near-sonic limit, in which nonlinear lattices can be modeled by the Korteweg-de Vries equation. The advance-delay-differential equations \eqref{Eqn:TW} can hence by replaced by a planar ODE, see \cite{FP99,FM02,FM03,Miz11,FM14,HML15} for
rigorous approximation formulas, \cite{FP02,FP04a,FP04b,HW08,HW09,Miz13} for stability proofs, \BMHC and \cite{DP14,GMWZ14,GP14} for generalizations to granular and polyatomic chains. \EMHC We also refer to \cite{IJ05,DHM06} for the small-amplitude limit, \BMHC to \cite{Jam12} for the case of near-binary interactions, \EMHC and to \cite{TV05,SZ09,TV10,SZ12,HMSZ13,TV14} for almost explicit traveling wave solutions in chains with piecewise linear force terms or small perturbations of such systems.
\par
Another asymptotic regime is the high-energy limit for chains with singular potential in which particle interactions can be interpreted as elastic collisions, see \cite{FM02,Tre04}. It has recently been shown in \cite{HM15} by a rigorous twoscale expansion that traveling waves in this limit are determined by an asymptotic ODE initial-value problem, which governs the behavior of the appropriately scaled distance profile. This observation provides a good starting point for further investigations 
because the limiting ODE problem yields the explicit scaling law for the wave speed as well as global approximation formulas for both the distance and the velocity profile. 
\par
In this paper we study the high-energy limit for potentials that do not possess a singularity but grow exponentially. Our overall strategy is similar to \cite{HM15} since we likewise employ local scalings in order to characterize the fine properties of the wave profiles near the critical positions and up to high accuracy. The details, however, are different because we have to impose other scaling relations and to deviate in the justification of the asymptotic estimates. Moreover, the final approximation formulas are more explicit because in our case the asymptotic ODE turns out to be integrable.
%
%
\subsection{Family of traveling waves and high-energy limit}
%
The literature contains many results that establish the existence
of periodic or solitary waves in FPU-type chains with fairly general interaction potential. We refer, for instance, to \cite{FW94,FV99,Her10} for constrained optimization, to \cite{SW97,Pan05} for critical-point techniques, and to \cite{IJ05} for center-manifold reduction. There is, however, no related uniqueness result although it is commonly believed that the different approaches yield -- at least for convex potentials -- the same family of waves. 
\par
In what follows we restrict our considerations to a periodic setting (the solitary case can be studied along the same lines) and construct traveling waves by the constrained maximization of a potential-energy functional subject to certain constraints. More precisely, we study a \emph{normalized velocity profile} $V$, which is supposed to be $2L$-periodic,
along with the optimization problem
\begin{align}
\label{Eqn:OptProb}
\text{Maximize} \quad \calP_\delta\at{V}\quad \text{subject to}\quad \norm{V}_2=1\quad \text{and}\quad\BMHC V\in \calC\EMHC\,, 
\end{align}
where the parameter $\delta>0$ becomes small in the high-energy-limit. The potential-energy functional $\calP_\delta$ is defined by
\begin{align}
\label{Eqn:DefOpA}
\calP_\delta\at{V}:=\int\limits_{-L}^{+L}\Phi\at{\frac{\bat{\calA V}\at{x}}{\delta}}\dint{x}\,,\qquad 
\bat{\calA V}\at{x}:=\int\limits_{x-\tfrac12}^{x+\tfrac12} V\at{\tilde{x}}\dint{\tilde{x}}\,,
\end{align} 
where the operator $\calA$ can be viewed as the convolution with the indicator function of the symmetric unit interval and provides the \emph{normalized distance profile} via $R:=\calA V$. Moreover, $\norm{\cdot}_p$ abbreviates for any $p\in\ccinterval{1}{\infty}$ the usual Lebesgue norm, evaluated on the periodicity cell $\ccinterval{-L}{+L}$, and the cone $\calC$ consists of all functions that are even, nonnegative, and unimodal on the periodicity cell, i.e.,
\begin{align}
\label{Eqn:DefCone}
\calC:=\Big\{ V\;:\; V\at{-x}=V\at{x}\geq V\at{\tilde{x}}\geq0 \;\;\text{for almost all}\;\;0\leq x\leq\tilde{x}\leq L \Big\}\,.
\end{align}
The key observation, see Proposition \ref{Prop:Existence} below, is that any solution $V_\delta$ to the constrained optimization problem \eqref{Eqn:OptProb} satisfies the integral equation
\begin{align}
\label{Eqn:TWInt}
\la_\delta^2 V_\delta=\frac{1}{\delta}\calA\Phi^\prime\at{\frac{\calA V_\delta}{\delta}}
\end{align}
for some multiplier $\la_\delta^2>0$, and  differentiation with respect to $x$ reveals
\begin{align}
\label{Eqn:ScaledTW}
R_\delta^\prime\at{x}=V_\delta\at{x+\tfrac12}-V_\delta\at{x-\tfrac12}\,,\qquad 
 V_\delta^\prime\at{x}=\frac{1}{\delta\la_\delta^2}
\at{\Phi^\prime\at{\frac{R_\delta\bat{x+\tfrac12}}{\delta}}
-
\Phi^\prime\at{\frac{R_\delta\bat{x-\tfrac12}}{\delta}}
}\,,
\end{align}
where $R_\delta$ is shorthand for $\calA V_\delta$. In particular, any maximizer of $\calP_\delta$ provides via
\begin{align}
\label{Eqn:PhysScalings}
\calV_\delta\at{x}:=\BMHC-\EMHC\la_\delta V_\delta\at{x}\,,\qquad \calR_\delta\at{x}:=
\frac{R_\delta\at{x}}{\delta}\,,\qquad \si_\delta:=\la_\delta\delta
\end{align}
a solution to \eqref{Eqn:TW} and hence a periodic traveling wave in the chain. It is, however, more convenient to work with the normalized profiles $V_\delta$ and $R_\delta$ because these converge in the high-energy limit $\delta\to0$ while $\calV_\delta$ and $\calR_\delta$ become unbounded.
\par
In this paper we fix $2<L<\infty$ and \BMHC assume that the force term $\Phi^\prime\at{r}$ grows like $r^\mu \exp\at{r}$, where $\mu\in\Rset$ is a given parameter.
A prototypical example (with $\mu=m$) is \EMHC 
\begin{align*}
\Phi^\prime\at{r}=\bat{r^{m}+c_{m-1}r^{m-1}+c_1r^1}\exp\at{r}\,,\qquad m\in\Nset\,,\qquad c_1,\,..,\,c_{m-1}\geq0
\end{align*}
\BMHC but our setting also covers (with $\mu=0$) the Toda case, in which the particles interact via\EMHC 
\begin{align}
\label{Eqn:TodaPot}
\Phi_{\mathrm{Toda}}^\prime\at{r}=\exp\at{r}-1\,.
\end{align} 
\BMHC More generally, in what follows we study the following class of FPU-type chains, where some constants have been normalized to ease the presentation.\EMHC
\begin{assumption}[properties of the force function] 
\label{Ass:Pot}
$\Phi^\prime$ can be written as
\begin{align*}
\Phi^\prime\at{r}=\Psi\at{r}\exp\at{r}\,,
\end{align*}
where $\Psi:\Rset_+\to\Rset_+$ is continuously differentiable with 
\begin{align*}
\Psi\at{0}=0
\end{align*}
and satisfies
\begin{align*}
\Psi\at{r}>0\,,\qquad
\Psi^\prime\at{r}+\Psi\at{r}\geq0\,,\qquad
\abs{\frac{\Psi\at{r}}{r^\mu}-1}\leq \frac{C}{r}\qquad \text{for all}\quad r>0
\end{align*}
and some constants $C>0$, $\mu\in\Rset$. In particular, the interaction potential $\Phi$ is convex on $\Rset_+$, attains a global minimum in $0$ and grows like $r^\mu\exp\at{r}$ as $r\to\infty$.
\end{assumption}
The asymptotic analysis presented below is based on the following result, which is illustrated in Figure \ref{Fig:Profiles} and proven in Appendix \ref{app} using arguments from \cite{Her10}. Notice that the convexity of $\Phi$ -- or equivalently, the monotonicity of $\Phi^\prime$ -- is needed to ensure the existence wave profiles within the cone $\calC$.
\begin{figure}[t]%
\centering{%
\includegraphics[width=0.95\textwidth]{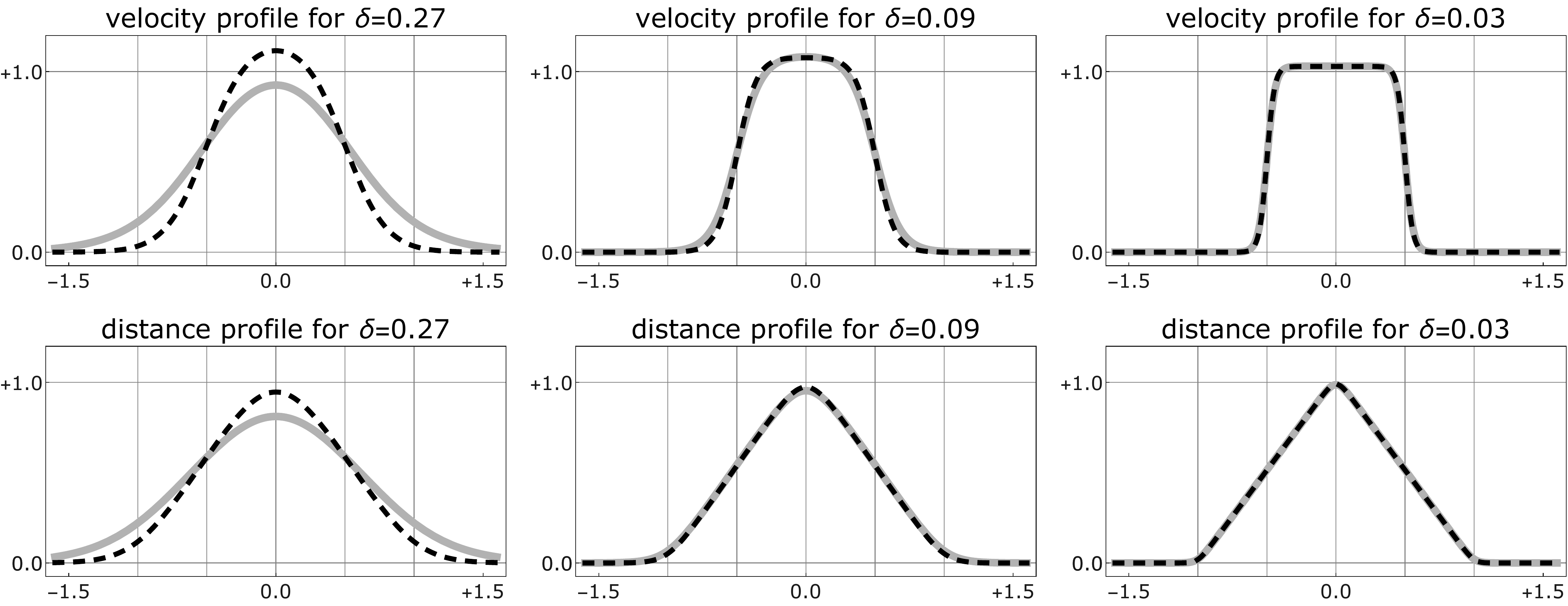}%
}%
\caption{ %
\emph{Top row}: Numerically computed velocity profile $V_\delta$ 
(black, dashed) along with the analytic approximation $\ol{V}_\delta$
from \eqref{Eqn:AppProfiles} (gray, solid) for $\Phi^\prime\at{r}=\at{r^2+2r}\exp\at{r}$ and three different values of $\delta$. In the high-energy limit $\delta\to0$, both $V_\delta$ and $\ol{V}_\delta$ converge to the indicator function $V_0$ but $\norm{V_\delta-\ol{V}_\delta}_p$ is asymptotically much smaller than  $\norm{V_\delta-V_0}_p$. \emph{Bottom row}: The corresponding distance profiles $R_\delta$ and $\ol{R}_\delta$, which approach the tent map $R_0$ as $\delta\to0$. The scaled approximation errors are plotted in Figure \ref{Fig:AppErrors} and the numerical scheme is briefly described at the end of Appendix \ref{app}.
} %
\label{Fig:Profiles}%
\end{figure} %
\begin{proposition}[family of traveling waves and convergence in the high-energy limit]
\label{Prop:Existence}%
For any $\delta>0$ there exist a maximizer $V_\delta$ and a Lagrangian multiplier $\la_\delta^2>0$ such that 
\begin{enumerate}
\item the traveling wave equations \eqref{Eqn:TWInt} and 
\eqref{Eqn:ScaledTW} are satisfied,
\item $V_\delta$ fulfills the norm constraints $\norm{V_\delta}_1=1$,
\item both $V_\delta$ and $R_\delta$ are $2L$-periodic, belong to $\calC$, and are pointwise positive.
\end{enumerate} 
Moreover, we have
\begin{align}
\label{Prop:Existence.Eqn1}
\bnorm{V_\delta-V_0}_2\;+\;
\bnorm{R_\delta-R_0}_\infty\quad\xrightarrow{\delta\to0}\quad 0\,,
\end{align}
where $V_0$ is the indicator function of the interval $\ccinterval{-\tfrac12}{+\tfrac12}$ and 
$R_0=\calA V_0$ is a tent map. 
\end{proposition}
%
%
\subsection{Main result and discussion}
%
%
In this paper we improve the convergence part of Proposition \ref{Prop:Existence}  by adapting the twoscale approach from \cite{HM15}. The key step is 
to study the convergence under the \emph{tip scaling} in \S\ref{sect:TipScal}, which describes the fine structure of $R_\delta$ in the vicinity of $x=0$ and \BMHC  reveals
the aforementioned asymptotic ODE initial-value problem. In our case, this problem reads\EMHC 
\begin{align}
\label{Eqn:LimitODEforS}
S_0^{\prime\prime}\at{y}= 2\exp\bat{-S_0\at{y}}\,,\qquad 
S_0^{\prime}\at{y}=S_0\at{y}=0\,,
\end{align}
\BMHC and admits the explicit solution
\begin{align}
\label{Eqn:limS}
S_0\at{y}=2\ln\bat{\cosh\at{y}}\,.
\end{align}
\EMHC
Afterwards we employ in \S\ref{sect:TranFootScal} the \emph{transition scaling} and the \emph{foot scaling} in order to characterize the jump-like behavior of the velocity profile $v_\delta$ near $x=\pm\tfrac12$ and the turn of the distance profiles near $x=\pm1$, respectively. In this way we identify the scaling law for the speed parameter $\la_\delta$ as well as refined asymptotic formulas for the normalized profiles, see \S\ref{sect:ScalLaws} and \S\ref{sect:Formulas}, respectively. Our main findings are formulated in Proposition \ref{Prop:ScalingRelations} and Theorem \ref{Thm:Approximation}, and can be summarized as follows.
\begin{result*} In the high-energy limit,  
the wave profiles from Proposition \ref{Prop:Existence} can be approximated by
\begin{align}
\label{Eqn:AppProfiles}
\ol{V}_\delta\at{x}:=\frac{\at{1+\delta}\sinh\at{\tfrac{1}{2\delta}}}{\cosh\at{\tfrac{1}{2\delta}}+\cosh\at{\frac{x}{\delta}}}\,,\qquad
\ol{R}_\delta\at{x}:=\at{\delta+\delta^2}\ln\at{1+\frac{\sinh^2\at{\tfrac{1}{2\delta}}}{\cosh^2\at{\tfrac{x}{2\delta}}}}
\end{align}
and we have 
\begin{align}
\label{Eqn:AppSpeed}
\la_\delta^2\delta^\mu\exp\at{-\frac{1}{\delta}}\quad \xrightarrow{\;\delta\to0\;}\quad \exp\at{1}\,.
\end{align}
More precisely, there exist a constant $C$ independent of $\delta$ such that
\begin{align}
\label{Eqn:AppResult}
\bnorm{V_\delta-\ol{V}_\delta}_p\leq C\delta \bnorm{V_\delta-V_0}_p\leq C\delta^{1+1/p}
\,,\qquad
\bnorm{R_\delta-\ol{R}_\delta}_p\leq C\delta \bnorm{R_\delta-R_0}_p\leq C\delta^2
\end{align}
holds for any $1\leq{p}\leq\infty$, where the norms have to be evaluated 
on the periodicity cell $\ccinterval{-L}{+L}$.
\end{result*}
A similar result holds in the case of solitary waves ($L=\infty$) but the corresponding proof is more technical, see the \BMHC related remarks to Theorem~\ref{Thm:Approximation} \EMHC and in Appendix \ref{app}. Our method can also be applied to other superpolynomially growing functions $\Phi^\prime$ but  the fundamental scaling relations and the proof details might be different. Unfortunately, the approach does not cover the high-energy limit of polynomial interaction potentials:
If $\Phi^\prime$ is a polynomial of degree $m$, the solution set to \eqref{Eqn:TWInt} exhibits asymptotically the invariance $\quadruple{\delta}{V}{R}{\la^2}\leftrightsquigarrow
\quadruple{\eta\delta}{\eta V}{\eta R}{\la^2/\eta}$ and the limit $\delta\to0$ is hence not governed by an ODE but by the advance-delay-differential equations \eqref{Eqn:ScaledTW} with $\delta=1$ and $\Phi^\prime\at{r}=r^m$. In this case, however, the arguments presented below can be adapted to derive 
asymptotic formulas for the limit $m\to\infty$.
\par
We also recall that \cite{HM15} concerns Lennard-Jones-type potentials with algebraic singularity of order $m>1$, which requires to modify the fundamental scaling relations and the details in the asymptotic analysis. Moreover, the corresponding asymptotic ODE  $S^{\prime\prime}_0\at{y}=c_m\at{1+S_0\at{y}}^{-m-1}$ cannot be integrated analytically, so the resulting approximation formulas are less explicit than \eqref{Eqn:AppProfiles}, and the guaranteed error bounds are not given by \eqref{Eqn:AppResult} but take another form and depend on $\delta^m$.
\par
\BMHC
The asymptotic ODE \eqref{Eqn:LimitODEforS} reveals that the dynamics of the distance $r_j=u_{j+1}-u_{j}$ between the particles $j$ and $j{+}1$ is for some times almost completely determined by $r_j$ itself and hence independent of $r_{j-1}$ and $r_{j+1}$. On the conceptional level, the high-energy limit can therefore be interpreted as the asymptotic regime in which the interactions between nearest neighbors can be neglected. The similar concept of an anticontinuum limit has been studied for a huge class of many-particle systems such as the discrete nonlinear Schr\"odinger equation, Klein-Gordan systems, and dimer chains (see, e.g. \cite{MKA94,FDMF02,BP13,CCCK14}). There is, however, an important difference: The decoupling, which manifests in the tip scaling, describes the effective dynamics of $r_j$ only for certain times, namely those corresponding to $-\tfrac12 < x < \tfrac12$, whereas the complete information on the waves can only be obtained by combining the different scalings. The foot scaling, for instance, reflects that for $-\tfrac32< x< -\tfrac12$ and $+\tfrac12< x<+\tfrac32$ the dynamics of $r_j$ is asymptotically governed by $r_{j+1}$ and $r_{j-1}$, respectively. 
\bigpar
We conclude the introduction with some comments on the Toda \BMHC chain \eqref{Eqn:TodaPot}. \EMHC For this integrable model, periodic traveling waves can be expressed in terms of Jacobi elliptic functions, see \cite{Tod67,Whi74}, and we readily justify that 
\begin{align*}
\calV_{\mathrm{Toda},\,\beta}\at{x}:=\BMHC-\EMHC\frac{2\sinh^2\at{\tfrac{1}{\beta}}}{\cosh\at{\tfrac{1}{\beta}}+\cosh\at{\frac{2x}{\beta}}}\,,\qquad
\calR_{\mathrm{Toda},\,\beta}\at{x}:=\ln\at{1+\frac{\sinh^2\at{\frac{1}{\beta}}}{\cosh^2\at{\frac{x} {\beta}}}}
\end{align*}
\BMHC along with \EMHC 
\begin{align*}
{\si}_{\mathrm{Toda},\,\beta}:=\beta\sinh\at{\tfrac{1}{\beta}}
\end{align*}
provide for any $\beta>0$ an exact solitary wave with \BMHC unimodal and \EMHC exponentially decaying wave profiles. We thus infer from \eqref{Eqn:AppResult} that the non-normalized profile functions $\calV_\delta$ and $\calR_\delta$ from \eqref{Eqn:PhysScalings} have for small $\delta>0$ basically the same shape as the solitary Toda wave with $\beta=2\delta$, but the  wave speed and the wave amplitudes scale differently since \eqref{Eqn:AppSpeed} depends on $\mu$, the leading algebraic order of the perturbation term $\Psi$. 
\par
\BMHC We also emphasize that for $\Phi^\prime=\Phi^\prime_{\mathrm{Toda}}$
the error \EMHC terms $\norm{V_\delta-\ol{V}_\delta}_p$ and $\norm{R_\delta-\ol{R}_\delta}_p$ are \emph{not}, as one might guess, exponentially small in $\delta$ but as large as the \BMHC related \EMHC bounds in \eqref{Eqn:AppResult}. This seems to be surprising but just means that the formulas in \eqref{Eqn:AppProfiles} are not optimal for the Toda chain. In fact, our approach applied to \eqref{Eqn:TodaPot} also provides high-energy approximations with exponential accuracy, see the discussion at the end of \S\ref{sect:Formulas}, but such formulas can directly be derived by asymptotic analysis of Jacobi's elliptic functions. For a nontrivial perturbation term $\Psi$, however, there are no explicit solution formulas and numerical results indicate that the estimates in \eqref{Eqn:AppResult} are optimal, see Figure \ref{Fig:AppErrors}. 
%
%
\section{Asymptotic analysis}
%
The starting point for our considerations is the
second-order advance-delay-differential equation
\begin{align}
\label{Eqn.SOADDE}
R_\delta^{\prime\prime}\at{x}=\frac{1}{\delta\la_\delta^2}
\at{\Phi^\prime\at{\frac{R_\delta\bat{x+1}}{\delta}}+
\Phi^\prime\at{\frac{R_\delta\bat{x-1}}{\delta}}-2
\Phi^\prime\at{\frac{R_\delta\bat{x}}{\delta}}
}\,,
\end{align}
which can be derived from \eqref{Eqn:ScaledTW} by eliminating the velocity profile $V_\delta$. It is also convenient to introduce the scalar quantities
\begin{align}
\label{Eqn:DefPrmsAandB}
a_\delta:=\frac{\delta}{R_\delta\at{0}}\,,\qquad b_\delta:=\sqrt{\delta^{2}a_\delta^{\mu}\la_\delta^2 \exp\at{-\frac{1}{a_\delta}}}\,,
\end{align}
which  appear naturally in many of the asymptotic formulas and estimates identified below. From \eqref{Prop:Existence.Eqn1} we immediately infer that the amplitude parameter $a_\delta$ satisfies
\begin{align}
\label{Eqn:LimitPrmA}
\frac{a_\delta}{\delta}\quad \xrightarrow{\delta\to0}\quad\frac{1}{R_0\at{0}}=1
\end{align}
but at this moment we do not know how $b_\delta$ depends on $\delta$. Below, however, we see that $b_\delta$ is an intrinsic length which determines the width of the transition layers in the velocity profile $V_\delta$ near $x=\pm\tfrac12$ as well as the fine properties of the distance profile $R_\delta$ near $x=0$ and $x=\pm1$. These observations finally imply $b_\delta=2\delta+\DO{\delta^2}$ and enable us to identify the scaling law for $\la_\delta$, see Proposition \ref{Prop:ScalingRelations}.
\bigpar
Throughout the paper, we denote by $c$ and $C$ generic positive constants -- being rather small and large, respectively -- which are independent of $\delta$ but can depend on the function $\Psi$ from Assumption~\ref{Ass:Pot} and the periodicity parameter $L$. In all proofs we further assume -- without saying so explicitly-- that $\delta$ is positive but sufficiently small. 
\subsection{Tip scaling}\label{sect:TipScal}
%
\begin{figure}[t]%
\centering{ %
\includegraphics[width=0.95\textwidth]{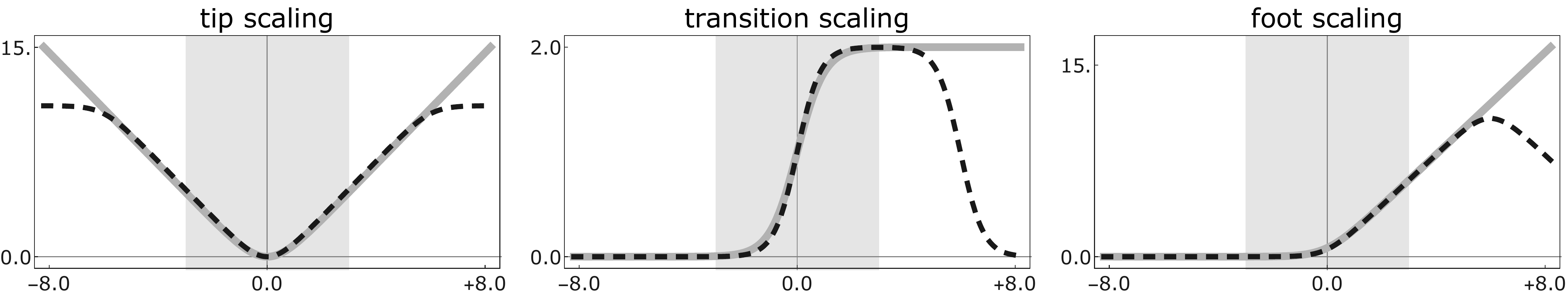} %
} %
\caption{ %
The scaled wave profiles $S_\delta$ (\emph{left}), $W_\delta$ (\emph{center}), and $T_\delta$ (\emph{right}) --  see \eqref{Eqn:DefTipScaling}, \eqref{Eqn:DefTranScaling}, and \eqref{Eqn:DefFootScaling} -- for the second simulation from Figure \ref{Fig:Profiles} as function of \BMHC the scaled space variable $y$ from \eqref{Eqn:ScaledSpace}. \EMHC The solid gray and dashed black lines represent the exact values for $\delta=0$ and the numerical data for $\delta=0.09$, respectively. The shaded region indicates the interval $I_\delta$ from \eqref{Eqn:DefIntI} and the corresponding approximation error is shown in Figure \ref{Fig:Errors}.%
} \label{Fig:Scaling}%
\end{figure} %
Our first goal is to describe the `tip of the tent', that is the 
asymptotic behavior of the distance profiles near $x=0$. To this end we introduce a scaled space variable $y\in\Rset$ by
\begin{align}
\label{Eqn:ScaledSpace}
x=b_\delta y
\end{align}
as well as the scaled distance profile $S_\delta$ via
\begin{align}
\label{Eqn:DefTipScaling}
S_\delta\at{y}:=\frac{1}{a_\delta}-\frac{1}{\delta}R_\delta\at{b_\delta y}\,,
\end{align}
which is even and satisfies
\begin{align}
\label{Eqn:IVforS}
S_\delta\at{0}=S_\delta^\prime\at{0}=0\,.
\end{align}
The key ingredient to our asymptotic analysis is to show that the scaled distance profile satisfies  
\begin{align*}
S_\delta^{\prime\prime}\at{y}\approx 2\exp\bat{-S_\delta\at{y}}
\end{align*}
on the interval
\begin{align}
\label{Eqn:DefIntI}
I_\delta=\left[-y_\delta^*,\,+y_\delta^*\right]\,,\qquad y_\delta^*:=\frac{1}{2b_\delta}\,,
\end{align}
which corresponds to the domain $\abs{x}\leq\tfrac12$. We can therefore approximate $S_\delta$ by \BMHC $S_0$ from \eqref{Eqn:limS}, which is the solution to the ODE initial-value problem \eqref{Eqn:LimitODEforS}. \EMHC The graphs of $S_\delta$ and $S_0$ are illustrated in the left panel of Figure \ref{Fig:Scaling}.
\begin{proposition}[convergence under the tip scaling]
\label{Prop:ConvTipScal}
There exists a constant $C$ such that
\begin{align}
\label{Prop:ConvTipScal.Eqn1}
\sup_{y\in I_\delta}\; \babs{S_\delta\at{y}-S_0\at{y}}+
\babs{S_\delta^{\prime}\at{y}-S_0^{\prime}\at{y}}+
\babs{S_\delta^{\prime\prime}\at{y}-S_0^{\prime\prime}\at{y}}\;\leq\; C \delta
\end{align}
and
\begin{align}
\label{Prop:ConvTipScal.Eqn2}
\babs{y_\delta^*S_\delta^\prime\at{y_\delta^*}-y_\delta^*S_0^\prime\at{y_\delta^*}}
\;\leq \; C\delta\,,\qquad 
\int\limits_{I_\delta} 
\babs{S_\delta^{\prime}\at{y}-S_0^{\prime}\at{y}}\dint{y}\;\leq\; C \delta
\end{align}
hold for all sufficiently small $\delta>0$.
\end{proposition}
\begin{proof} 
\ul{\emph{ Preliminaries}}: Since $S_\delta$ is even due to $R_\delta\in\calC$ and \eqref{Eqn:DefTipScaling}, it is sufficient to consider values $y$ with 
\begin{align}
\label{Prop:ConvTipScal.PEqn21}
0\leq y\leq y_\delta^*\,.
\end{align}
The convergence in \eqref{Prop:Existence.Eqn1} and the unimodality of $R_\delta$ 
ensure
\begin{align}
\label{Prop:ConvTipScal.PEqn2}
\tfrac14\leq R_\delta\at{\pm\tfrac12}\leq\tfrac34 \,,\qquad 
\tfrac34\leq R_\delta\at{0}\leq\tfrac54 
\end{align}
as well as
\begin{align}
\label{Prop:ConvTipScal.PEqn1}
R_\delta\at{b_\delta y\pm 1}\leq R_\delta\at{\pm\tfrac12} \leq 
R_\delta\at{b_\delta y}\leq R_\delta\at{0}
\end{align}
for all $y$ under consideration. In particular, $S_\delta$ is nonnegative on $I_\delta$.
\par
\ul{\emph{Perturbed ODE for $S_\delta$}}: Combining \eqref{Eqn:DefTipScaling} with the second-order advance-delay differential equation \eqref{Eqn.SOADDE} we obtain
\begin{align}
\label{Prop:ConvTipScal.PEqn9}
\frac{\delta^2 \la_\delta^2}{b_\delta^2} S_\delta^{\prime\prime}\at{y}=
2\Phi^\prime\at{\frac{R_\delta\bat{b_\delta{y}}}{\delta}}
-\Phi^\prime\at{\frac{R_\delta\bat{b_\delta{y}-1}}{\delta}}-
\Phi^\prime\at{\frac{R_\delta\bat{b_\delta{y}+1}}{\delta}}
\end{align}
and infer from \eqref{Prop:ConvTipScal.PEqn1} that $S_\delta$ is convex on $I_\delta$ ($S_0$ is convex on $\Rset$). \BMHC Due to \eqref{Eqn:DefTipScaling} and since \eqref{Eqn:DefPrmsAandB} implies \EMHC
\begin{align}
\label{Prop:ConvTipScal.PEqn10}
\frac{b_\delta^2}{\delta^{2}\la_\delta^2}=a_\delta^{\mu}\exp\at{-\frac{1}{a_\delta}}
\,,
\end{align}
\BMHC equation \eqref{Prop:ConvTipScal.PEqn9} can be written as \EMHC
\begin{align}
\label{Prop:ConvTipScal.PEqn3}
S_\delta^{\prime\prime}\at{y}=2\frac{b_\delta^2}{\delta^2\la_\delta^2}\Phi^\prime\at{\frac{R_\delta\at{b_\delta{y}}}{\delta}}+f_\delta\at{y}=2\exp\Bat{-S_\delta\at{y}}+f_\delta\at{y} +g_\delta\at{y}\,,
\end{align}
\BMHC where the error terms are given by \EMHC
\begin{align*}
f_\delta\at{y}:=-
\frac{b_\delta^2}{\delta^2\la_\delta^2}\at{\Phi^\prime\at{\frac{R_\delta\at{b_\delta y-1}}{\delta}}+\Phi^\prime\at{\frac{R_\delta\at{b_\delta y+1}}{\delta}}}
\end{align*}
and
\begin{align}
\label{Prop:ConvTipScal.PEqn18}
g_\delta\at{y}:=2\exp\Bat{-S_\delta\at{y}}\at{a_\delta^\mu\Psi\at{\frac{R_\delta\at{b_\delta y}}{\delta}}-1}\,.
\end{align}
\ul{\emph{Pointwise estimates for $f_\delta$ and $g_\delta$}}: The first error term can -- due to \eqref{Eqn:LimitPrmA}, \eqref{Prop:ConvTipScal.PEqn2}, \eqref{Prop:ConvTipScal.PEqn1}, and  \eqref{Prop:ConvTipScal.PEqn10} -- be bounded by
\begin{align}
\label{Prop:ConvTipScal.PEqn5}
\begin{split}
\sup_{y\in I_\delta}\babs{f_\delta\at{y}}&= a_\delta^\mu \exp\at{-\frac{1}{a_\delta}}\at{\Phi^\prime\at{\frac{R_\delta\at{b_\delta y-1}}{\delta}}+\Phi^\prime\at{\frac{R_\delta\at{b_\delta y+1}}{\delta}}}\\&
\leq C\delta^{\mu} \exp\at{-\frac{1}{a_\delta}} \Phi^\prime\at{\frac{3}{4 \delta}} \leq C\exp\at{-\frac{c}{\delta}}\,,
\end{split}
\end{align}
where we also used the monotonicity and the asymptotic behavior of $\Phi^\prime$. To estimate the second error term, we observe that
\begin{align*}
a_\delta^\mu\Psi\at{\frac{R_\delta\at{b_\delta y}}{\delta}}-1=
\at{a_\delta^\mu\at{\frac{R_\delta\at{b_\delta y}}{\delta}}^\mu-1}+
a_\delta^\mu\at{\frac{R_\delta\at{b_\delta y}}{\delta}}^\mu\at{\frac{
\Psi\at{\displaystyle\frac{R_\delta\at{b_\delta y}}{\delta}}}{\at{\displaystyle\frac{R_\delta\at{b_\delta y}}{\delta}}^\mu}-1}
\,
\end{align*}
and with \eqref{Eqn:DefTipScaling} we conclude
\begin{align*}
\at{a_\delta^\mu\at{\frac{R_\delta\at{b_\delta y}}{\delta}}^\mu-1}=
\Bat{1-a_\delta S_\delta\at{y}}^\mu-1\,.
\end{align*}
Moreover, \eqref{Eqn:LimitPrmA} combined with \eqref{Eqn:DefTipScaling} and \eqref{Prop:ConvTipScal.PEqn1} guarantees
\begin{align*}
c \leq 1-a_\delta S_\delta\at{y}=\frac{a_\delta}{\delta}R_\delta\at{b_\delta y}\leq C\,,
\end{align*}
so \BMHC a simple Taylor argument provides \EMHC
\begin{align*}
\abs{a_\delta^\mu\at{\frac{R_\delta\at{b_\delta y}}{\delta}}^\mu-1}\leq Ca_\delta S_\delta\at{y}\leq C\delta S_\delta\at{y}\,.
\end{align*}
On the other hand, thanks to Assumption \ref{Ass:Pot}, \eqref{Eqn:LimitPrmA}, and \eqref{Prop:ConvTipScal.PEqn2}+\eqref{Prop:ConvTipScal.PEqn1} we  estimate
\begin{align*}
a_\delta^\mu\at{\frac{R_\delta\at{b_\delta y}}{\delta}}^\mu\abs{\frac{
\Psi\at{\displaystyle\frac{R_\delta\at{b_\delta y}}{\delta}}}{\at{\displaystyle\frac{R_\delta\at{b_\delta y}}{\delta}}^\mu}-1}\leq Ca_\delta^\mu\at{\frac{ R_\delta\at{b_\delta y}}{\delta}}^{\mu-1}\leq C\delta\,,
\end{align*}
and summing the last two results we deduce from \eqref{Prop:ConvTipScal.PEqn18} that
\begin{align}
\label{Prop:ConvTipScal.PEqn4}
\sup_{y\in I_\delta}\babs{g_\delta\at{y}} \leq C \delta \exp\Bat{-S_\delta\at{y}}\Bat{1+S_\delta\at{y}}\leq 
C \delta \exp\Bat{-cS_\delta\at{y}}\,.
\end{align}
\ul{\emph{Upper bound for $b_\delta$}}: Employing \BMHC \eqref{Eqn:LimitPrmA} as well as \EMHC the properties of $\Phi^\prime$, $R_\delta$, and $\calA$ we show that
\begin{align*}
\frac{1}{\delta}\at{\calA \Phi^\prime\at{\frac{R_\delta\BMHC\at{x}\EMHC}{\delta}}}\at{x}\leq \frac{1}{\delta}\Phi^\prime\at{\frac{1}{a_\delta}}\leq \frac{C}{\delta^{\mu+1}}\exp\at{\frac{1}{a_\delta}}\quad \text{for all}\quad x\,,
\end{align*}
and testing the traveling wave equation \eqref{Eqn:TWInt} with $V_0$ gives
\begin{align*}
c\la_\delta^2\leq \la_\delta^2 \int\limits_{-1/2}^{1/2} V_\delta\at{x}V_0\at{x}\dint{x}\leq 
\frac{C}{\delta^{\mu+1}}\exp\at{\frac{1}{a_\delta}}
\end{align*} 
thanks to $V_\delta\to V_0$ and $\int\limits_{-1/2}^{+1/2}V_0\at{x}^2\dint{x}=1$, see Proposition \ref{Prop:Existence}. The estimate
\begin{align}
\label{Eqn:UpperBoundPrmB}
b_\delta\leq C\sqrt{\delta}
\end{align}
is therefore granted by \eqref{Eqn:DefPrmsAandB} and \eqref{Eqn:LimitPrmA}. In particular, the interval $I_\delta$ is  large and we have
\begin{align}
\label{Prop:ConvTipScal.PEqn20}
y_\delta^*>1\,.
\end{align}
\ul{\emph{Local approximation estimates}}: We define
\begin{align*}
d_\delta\at{y}:=\babs{S_\delta\at{y}-S_0\at{y}}+\babs{S_\delta^\prime\at{y}-S_0^\prime\at{y}}
\end{align*} 
and recall that
\begin{align}
\label{Prop:ConvTipScal.PEqn16}
\babs{\exp\bat{-S_\delta\at{y}}-\exp\bat{-S_0\at{y}}}\leq e_\delta\at{z}
\babs{S_\delta\at{y}-S_0\at{y}}\,,
\end{align} 
where
\begin{align*}
e_{\delta}\at{z}:=\exp\Bat{-\min{\big\{S_\delta\at{z},\, S_0\at{z}\big\}}}\,.
\end{align*} 
Moreover, by standard arguments we derive from the ODE initial-value problems \eqref{Eqn:LimitODEforS} and \eqref{Eqn:IVforS}+\eqref{Prop:ConvTipScal.PEqn3} the Gronwall-type inequality
\begin{align}
\label{Prop:ConvTipScal.PEqn8}
d_\delta\at{y}\leq C\int\limits_0^y \Bat{ e_{\delta}\at{\tilde{y}} d_\delta\at{\tilde{y}}+f_\delta\at{\tilde{y}}+ g_\delta\at{\tilde{y}}}\dint{\tilde{y}}\,,
\end{align}
which controls the growth of $d_\delta$. More precisely, for $0\leq y\leq 1$
and using \eqref{Prop:ConvTipScal.PEqn5}, \eqref{Prop:ConvTipScal.PEqn4}, and \eqref{Prop:ConvTipScal.PEqn20} we estimate
\begin{align*}
e_{\delta}\at{y}\leq 1\,,\qquad f_\delta\at{y}\leq  C\delta\,,\qquad g_\delta\at{y}\leq C\delta\,,
\end{align*}
so \eqref{Prop:ConvTipScal.PEqn8} implies
\begin{align*}
d_\delta\at{y}\leq C\int\limits_0^y d_\delta\at{\tilde{y}}\dint{\tilde{y}}+C\delta=:D_{\mathrm{loc},\,\delta}\at{y}
\end{align*}
and hence 
\begin{align*}
D_{\mathrm{loc},\,\delta}^\prime\at{y}=C d_\delta\at{y}\leq CD_{\mathrm{loc},\,\delta}\at{y}\,,\qquad
D_{\mathrm{loc},\,\delta}\at{0}\leq C\delta
\end{align*}
for all $0\leq y\leq 1$. The comparison principle for scalar ODE thus \BMHC ensures \EMHC 
\begin{align*}
\sup_{0\leq y\leq 1}d_\delta\at{y}\leq\sup_{0\leq y\leq 1}D_{\mathrm{loc},\,\delta}\at{y}\leq  C \delta\,,
\end{align*}
and this gives
\begin{align}
\label{Prop:ConvTipScal.PEqn6}
\tfrac32 S_0\at{1}\geq S_\delta\at{1}\geq \tfrac12 S_0\at{1}>0\,,\qquad  S_\delta^\prime\at{1}\geq \tfrac12 S_0^\prime\at{1}>0\,.
\end{align}
\ul{\emph{Lower bound for $b_\delta$}}: Combining \eqref{Prop:ConvTipScal.PEqn6} with the convexity of $S_\delta$ on $I_\delta$ we prove
\begin{align}
\label{Prop:ConvTipScal.PEqn11}
S_\delta\at{y}\geq \max\big\{0, S_\delta\at{1}+S_\delta^\prime\at{1}\at{y-1}\big\}\geq
\max\big\{0, cy-C\big\}
\end{align}
for all $y$ with \eqref{Prop:ConvTipScal.PEqn21}, and evaluating \eqref{Eqn:DefTipScaling} for $y=y_\delta^*$ we establish -- using also \eqref{Prop:ConvTipScal.PEqn2}+\eqref{Prop:ConvTipScal.PEqn20} -- the inequalities
\begin{align*}
\frac{c}{b_\delta}\leq S_\delta\at{y_\delta^*}=\frac{R_\delta\at{0}-R_\delta\at{\tfrac12}}{\delta}\leq
\frac{C}{\delta}\,.
\end{align*}
Rearranging terms we obtain
\begin{align}
\label{Eqn:LowerBoundPrmB}
b_\delta \geq c \delta
\end{align} 
and hence
\begin{align}
\label{Prop:ConvTipScal.PEqn7}
\int\limits_0^y \at{1+\tilde{y}}\babs{ f_\delta\at{\tilde{y}}}\dint{\tilde{y}}\leq
C\exp\at{-\frac{c}{\delta}}\int\limits_{0}^{\BMHC y^*_\delta\EMHC }\at{1+\tilde{y}}\dint{\tilde{y}}=
C\exp\at{-\frac{c}{\delta}}\at{\frac{1}{b_\delta}+\frac{1}{b_\delta^2}}\leq C\delta
\end{align}
thanks to the upper bound for $f_\delta$ from \eqref{Prop:ConvTipScal.PEqn5}. Moreover, in view of \eqref{Prop:ConvTipScal.PEqn4}, the lower bound \eqref{Prop:ConvTipScal.PEqn11} also implies
\begin{align}
\label{Prop:ConvTipScal.PEqn12}
\int\limits_0^y \at{1+\tilde{y}}\babs{g_\delta\at{\tilde{y}}}\dint{\tilde{y}}\leq C\delta
\int\limits_0^\infty\at{1+\tilde{y}}\exp\bat{-c\tilde{y}}\dint{\tilde{y}}\leq C\delta\,.
\end{align}
\ul{\emph{Pointwise estimates on the interval $I_\delta$}}: %
By \eqref{Prop:ConvTipScal.PEqn11} -- which also holds for $S_0$ -- we have
\begin{align}
\label{Prop:ConvTipScal.PEqn14}
e_{\delta}\at{y}\leq C\exp\at{-c y}
\end{align}
and combining this with \eqref{Prop:ConvTipScal.PEqn7}+\eqref{Prop:ConvTipScal.PEqn12}, we can simplify the Gronwall estimate \eqref{Prop:ConvTipScal.PEqn8} to
\begin{align*}
d_\delta\at{y}\leq C\int\limits_0^y \exp\at{-c\tilde{y}} d_\delta\at{\tilde{y}}\dint{\tilde{y}}+ C \delta=:D_{\mathrm{glob},\,\delta}\at{y}\,.
\end{align*}
By direct computations we verify
\begin{align*}
D_{\mathrm{glob},\,\delta}^\prime\at{y}= C\exp\at{-cy}d_\delta\at{y}\leq C\exp\at{-cy}D_{\mathrm{glob},\,\delta}\at{y}\,,\qquad D_{\mathrm{glob},\,\delta}\at{0}\leq C\delta
\end{align*}
and infer from the comparison principle for scalar ODEs the global error estimate
\begin{align}
\label{Prop:ConvTipScal.PEqn15}
\sup_{0\leq y\leq y_\delta^*} d_\delta\at{y}\leq \sup_{0\leq y\leq y_\delta^*} D_{\mathrm{glob},\,\delta}\at{y}\leq C\delta\exp\at{C\int\limits_0^\infty \exp\at{-c\tilde{y}}\dint{\tilde{y}}}\leq C \delta\,,
\end{align}
\BMHC which yields \EMHC the desired pointwise estimates for $S_\delta-S_0$ and $S_\delta^\prime-S_0^\prime$ \BMHC on the interval $I_\delta$. \EMHC The corresponding thesis for the second derivatives is a direct consequence of the estimates derived so far and the ODEs  \eqref{Eqn:LimitODEforS}+\eqref{Prop:ConvTipScal.PEqn3}. In particular, we have shown \eqref{Prop:ConvTipScal.Eqn1}.
\par
\ul{\emph{Further estimates on $I_\delta$}}: %
We define 
\begin{align*}
G_\delta\at{y}:= yS_\delta^\prime\at{y}-S_\delta\at{y}
\end{align*}
and employ \eqref{Eqn:LimitODEforS}, \eqref{Prop:ConvTipScal.PEqn3}, and \eqref{Prop:ConvTipScal.PEqn16} to \BMHC obtain \EMHC 
\begin{align*}
\int\limits_0^{y_\delta^*} \babs{G_\delta^\prime\at{y}-G_0^\prime\at{y}}\dint{y}&=
\int\limits_0^{y_\delta^*}\babs{yS_\delta^{\prime\prime}\at{y}-yS_0^{\prime\prime}\at{y}}\dint{y}\\& \leq\int\limits_0^{y_\delta^*} y\Bat{e_\delta\at{y} d_\delta\at{y}+\babs{f_\delta\at{y}}+\babs{g_\delta\at{y}}}\dint{y}\,.
\end{align*}
Inserting 
\eqref{Prop:ConvTipScal.PEqn7}, \eqref{Prop:ConvTipScal.PEqn12}, \eqref{Prop:ConvTipScal.PEqn14}, and \eqref{Prop:ConvTipScal.PEqn15} gives
\begin{align}
\label{Prop:ConvTipScal.PEqn17}
\int\limits_0^{y_\delta^*}\babs{G_\delta^\prime\at{y}-G_0^\prime\at{y}}\dint{y}\leq C \delta\,,
\end{align}
and in view of $G_\delta\at{0}=G_0\at{0}$ we arrive at
\begin{align*}
\sup\limits_{0\leq y\leq y_\delta^*} \babs{G_\delta\at{y}-G_0\at{y}}\leq C\delta\,,
\end{align*}
which implies in combination with \eqref{Prop:ConvTipScal.Eqn1} the first claim in \eqref{Prop:ConvTipScal.Eqn2}. Moreover, integrating the identity
\begin{align*}
S_\delta^\prime\at{y}-S_0^\prime\at{y}=S_\delta^\prime\at{y_\delta^*}-S_0^\prime\at{y_\delta^*}-\int\limits_{y}^{y_\delta^*}
S_\delta^{\prime\prime}\at{\tilde{y}}-S_0^{\prime\prime}\at{\tilde{y}}
\dint{\tilde{y}}
\end{align*}%
we obtain
\begin{align*}
\int\limits_{0}^{y_\delta^*}\babs{S_\delta^\prime\at{y}-S_0^\prime\at{y}}\dint{y}
&\leq y_\delta^*\abs{S_\delta^\prime\at{y_\delta^*}-S_0^\prime\at{y_\delta^*}}+
\int\limits_{0}^{y_\delta^*}
\babs{yS_\delta^{\prime\prime}\at{y}-yS_0^{\prime\prime}\at{y}}
\dint{y}\\
&= \abs{y_\delta^*S_\delta^\prime\at{y_\delta^*}-y_\delta^*S_0^\prime\at{y_\delta^*}}+
\int\limits_{0}^{y_\delta^*}
\babs{G_\delta^{\prime}\at{y}-G_0^{\prime}\at{y}}
\dint{y}\,,
\end{align*}
so \eqref{Prop:ConvTipScal.Eqn2}$_2$ follows from \eqref{Prop:ConvTipScal.Eqn2}$_1$ and
\eqref{Prop:ConvTipScal.PEqn17}.
\end{proof}
\begin{figure}[t] %
\centering{ %
\includegraphics[width=0.95\textwidth]{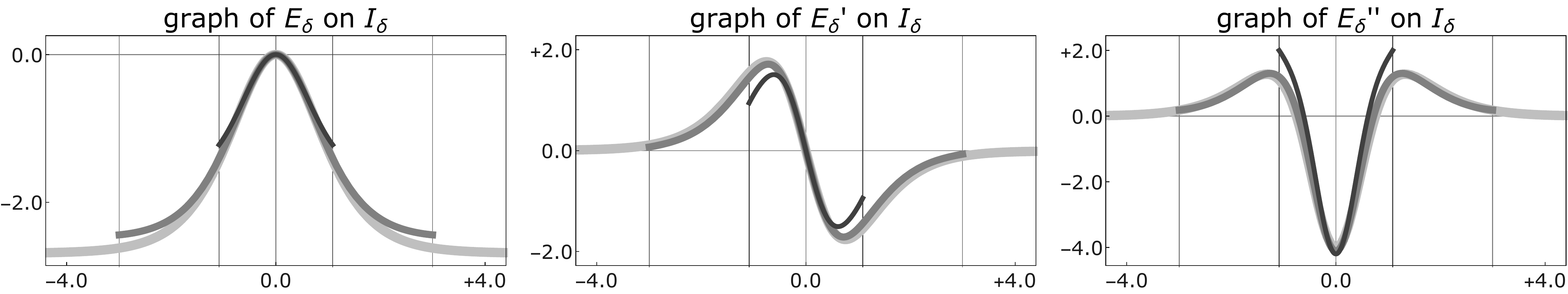} %
} %
\caption{ %
Numerical values of the scaled approximation error $E_\delta\at{y}:=\delta^{-1}\at{S_\delta\at{y}-S_0\at{y}}$ on the interval $I_\delta$ for $\delta=0.27$ (black), $\delta=0.09$ (dark gray), and $\delta=0.03$ (light gray). The three panels correspond to different derivatives and provide numerical evidence for the optimality of the pointwise estimates in \eqref{Prop:ConvTipScal.Eqn1}.
} %
\label{Fig:Errors}
\end{figure} %
The proof of Proposition \ref{Prop:ConvTipScal} reveals that the estimates \eqref{Prop:ConvTipScal.Eqn1} and \eqref{Prop:ConvTipScal.Eqn2} can -- at least in principle -- be improved as follows. Assuming an asymptotic expansion
\begin{align}
\label{Eqn:AsympExp}
S_\delta = S_0+\delta S_1  +...\,,\qquad \BMHC a_\delta = \delta + ...\,,\EMHC
\end{align}
the identities \eqref{Eqn:LimitODEforS}+\eqref{Eqn:DefTipScaling}+\eqref{Prop:ConvTipScal.PEqn3} provide a linear but inhomogeneous and non-autonomous ODE for the coefficient function $S_1$, which \BMHC here is \EMHC not related to $\delta=1$ but denotes $\tfrac{\dint}{\dint \delta}  S_\delta|_{\delta=0}$.  For instance, in the case $\Phi^\prime\at{r}= r^\mu\exp\at{r}$ we \BMHC find \EMHC
\begin{align*}
S_\delta^{\prime\prime}\at{y}= 2
\bat{1+a_\delta S_\delta\at{y}}^\mu\exp\bat{-S_\delta\at{y}}\;+\;f_\delta\at{y}
\end{align*}
and obtain
\begin{align*}
S_1^{\prime\prime}\at{y}= 2\BMHC\Bat{ \mu S_0\at{y}-S_1\at{y}}\EMHC\exp\bat{-S_0\at{y}}\,,\qquad 
S_1\at{0}=S_1^{\prime}\at{0}=0
\end{align*}
as equation for the next-to-leading-order term \BMHC thanks to \eqref{Prop:ConvTipScal.PEqn5}+\eqref{Eqn:AsympExp}. \EMHC Numerical approximations of $S_1$ can be found in Figure \ref{Fig:Errors} for the data from Figure \ref{Fig:Profiles}.
\par
We finally mention that the convergence under the tip scaling is considerably better for the Toda chain \eqref{Eqn:TodaPot}. In fact, $g_\delta\at{y}$ from \eqref{Prop:ConvTipScal.PEqn18} is for any $y\in I_\delta$ much smaller than $\DO{\delta}$ and thus we show that the error bounds in \eqref{Prop:ConvTipScal.Eqn1} and \eqref{Prop:ConvTipScal.Eqn2} are exponentially small in $\delta$.
%
%

\subsection{Transition and foot scaling}\label{sect:TranFootScal}
%
We next characterize the transition layer of the velocity profile $V_\delta$ near $x=-\tfrac12$ as well as the behavior of the distance profile $R_\delta$ near $x=-1$ (`foot of the tent') by studying the functions
\begin{align}
\label{Eqn:DefTranScaling}
W_\delta\at{y}:=\frac{b_\delta}{\delta} V_\delta\bat{-\tfrac12+b_\delta y}
\end{align}
and
\begin{align}
\label{Eqn:DefFootScaling}
T_\delta\at{y}:=\frac{1}{\delta}R_\delta\at{-1+b_\ga y}\,,
\end{align}
respectively. In particular, we employ our results on the tip scaling and show 
that $W_\delta$ and $T_\delta$ behave on the interval $I_\delta$ like
\begin{align}
\label{Eqn:LimW}
W_0\at{y}:=\tfrac12 S_0^\prime\at{y}+1=1+\tanh\at{y}
\end{align}
and
\begin{align}
\label{Eqn:limT}
T_0\at{y}:=\tfrac12 S_0\at{y}+y+\ln{2}=\ln\bat{2\cosh\at{y}}+y\,.
\end{align}
As illustrated in Figure \ref{Fig:Scaling}, $W_\delta|_{I_\delta}$ describes $V_\delta$ on the domain $-1\leq x\leq0$, while $T_\delta|_{I_\delta}$ corresponds to $R_\delta$ restricted to $-\tfrac32\leq x\leq-\tfrac12$.
\begin{proposition}[convergence under the transition and the tip scaling]
\label{Prop:ConvTranScal}
There exists a constant $C$ such that
\begin{align}
\label{Prop:ConvTranScal.Eqn1}
\sup_{y\in I_\delta}\; \babs{W_\delta\at{y}-W_0\at{y}}+
\babs{W_\delta^{\prime}\at{y}-W_0^{\prime}\at{y}}\;\leq\; C \delta\,,\qquad
\int\limits_{I_\delta}\babs{W_\delta\at{y}-W_0\at{y}}\dint{y}\;\leq\; C\delta
\end{align}
and 
\begin{align}
\label{Prop:ConvTranScal.Eqn2}
\sup_{y\in I_\delta}\; \babs{T_\delta\at{y}-T_0\at{y}}+
\babs{T_\delta^{\prime}\at{y}-T_0^{\prime}\at{y}}+
\babs{T_\delta^{\prime\prime}\at{y}-T_0^{\prime\prime}\at{y}}\;\leq\; C \delta
\end{align}
hold for all sufficiently small $\delta>0$.
\end{proposition}
\begin{proof} 
\ul{\emph{Formulas for $W_\delta$ and $T_\delta$}}: Definition \eqref{Eqn:DefTranScaling} and the traveling wave equation \eqref{Eqn:TWInt} provide
\begin{align*}
W_\delta\at{y}=
\frac{b_\delta}{\delta^2\la_\delta^2}\int\limits_{b_\delta y-1}^{b_\delta y}\Phi^\prime\at{\frac{R_\delta\at{x}}{\delta}}\dint{x}\,,
\end{align*}
and after splitting off the dominant contributions to the integral (coming from $\abs{x}\leq\tfrac12$ ) and inserting the second-order advance-delay differential relation \eqref{Eqn.SOADDE} we obtain
\begin{align}
\label{Prop:ConvTranScal.PEqn0}
W_\delta\at{y}=-\frac{b_\delta}{2\delta}\int\limits_{-1/2}^{b_\delta y}R_\delta^{\prime\prime}\at{x}\dint{x}+ f_\delta\at{y}
\end{align}
with 
\begin{align*}
f_\delta\at{y}:=
\frac{b_\delta}{2\delta^2\la_\delta^2}\int\limits_{-1/2}^{b_\delta y}\at{\Phi^\prime\at{\frac{R_\delta\bat{x-1}}{\delta}}+\Phi^\prime\at{\frac{R_\delta\bat{x+1}}{\delta}}}\dint{x}
+\frac{b_\delta}{\delta^2\la_\delta^2}\int\limits_{b_\delta y-1}^{-1/2}\Phi^\prime\at{\frac{R_\delta\bat{x}}{\delta}}\dint{x}
\end{align*}
and
\begin{align*}
f_\delta^\prime\at{y}=\frac{b_\delta^2}{2\delta^2\la_\delta^2}\at{\Phi^\prime\at{\frac{R_\delta\bat{b_\delta y-1}}{\delta}}+\Phi^\prime\at{\frac{R_\delta\bat{b_\delta y+1}}{\delta}}}-\frac{b_\delta^2}{\delta^2\la_\delta^2}\Phi^\prime\at{\frac{R_\delta\bat{b_\delta y-1}}{\delta}}\,.
\end{align*}
Similarly, we derive from \BMHC \eqref{Eqn:TWInt} and \eqref{Eqn:DefFootScaling} \EMHC the identity
\begin{align*}
T_\delta\at{y}=\frac{1}{\delta^2\la_\delta^2}\int\limits_{b_\delta y-\tfrac32}^{b_\delta y-\tfrac12}\int\limits_{x-\tfrac12}^{x+\tfrac12}
\Phi^\prime\at{\frac{R_\delta\bat{\tilde{x}}}{\delta}}
\dint{\tilde{x}}\dint{x}\,,
\end{align*}
which can be written as 
\begin{align}
\label{Prop:ConvTranScal.PEqn1}
\begin{split}
T_\delta\at{y}&=\frac{1}{\delta^2\la_\delta^2}\int\limits_{-1}^{b_\delta y-\tfrac12}\int\limits_{-\tfrac12}^{x+\tfrac12}
\Phi^\prime\at{\frac{R_\delta\bat{\tilde{x}}}{\delta}}
\dint{\tilde{x}}\dint{x}+g_\delta\at{y}
\\&=-\frac{1}{2\delta}\int\limits_{-1}^{b_\delta y-\tfrac12}\int\limits_{-\tfrac12}^{x+\tfrac12}
R_\delta^{\prime\prime}\bat{\tilde{x}}
\dint{\tilde{x}}\dint{x}+g_\delta\at{y}+h_\delta\at{y}\,,
\end{split}
\end{align}
where the corresponding error terms are given by 
\begin{align*}
g_\delta\at{y}:=\frac{1}{\delta^2\la_\delta^2}\int\limits_{b_\delta y-\tfrac32}^{-1}\int\limits_{x-\tfrac12}^{x+\tfrac12}
\Phi^\prime\at{\frac{R_\delta\bat{\tilde{x}}}{\delta}}
\dint{\tilde{x}}\dint{x}+\frac{1}{\delta^2\la_\delta^2}\int\limits_{-1}^{b_\delta y-\tfrac12}\int\limits_{x-\tfrac12}^{-\tfrac12}
\Phi^\prime\at{\frac{R_\delta\bat{\tilde{x}}}{\delta}}
\dint{\tilde{x}}\dint{x}
\end{align*}
and
\begin{align*}
h_\delta\at{y}:=\frac{1}{2\delta^2\la_\delta^2}\int\limits_{-1}^{b_\delta y-\tfrac12}\int\limits_{-\tfrac12}^{x+\tfrac12}
\at{\Phi^\prime\at{\frac{R_\delta\bat{\tilde{x}-1}}{\delta}}
+\Phi^\prime\at{\frac{R_\delta\bat{\tilde{x}+1}}{\delta}}}
\dint{\tilde{x}}\dint{x}\,.
\end{align*}
\ul{\emph{Pointwise estimates for the error terms}}:
Employing the unimodality of $R_\delta$ as well as \eqref{Prop:Existence.Eqn1} we derive
\begin{align*}
\sup\limits_{\abs{x}\geq 1/2}\Phi^\prime\at{\frac{R_\delta\at{x}}{\delta}}\leq \Phi^\prime\at{\frac{R_\delta\at{\pm1/2}}{\delta}}\leq \Phi^\prime\at{\frac{3}{4\delta}}\leq \frac{C}{\delta^\mu}\exp\at{\frac{3}{4\delta}}\,,
\end{align*}
and infer from \eqref{Eqn:DefPrmsAandB} the estimate
\begin{align*}
\sup\limits_{\abs{x}\geq 1/2}\frac{b_\delta^2}{\delta^2\la_\delta^2 }\Phi^\prime\at{\frac{R_\delta\at{x}}{\delta}}
\leq C\frac{a_\delta^\mu}{\delta^\mu}\exp\at{-\frac{1}{a_\delta}}\exp\at{\frac{3}{4\delta}}\,.
\end{align*}
Consequently, and in view of \eqref{Eqn:LimitPrmA} and \eqref{Eqn:LowerBoundPrmB} we conclude
\begin{align}
\label{Prop:ConvTranScal.PEqn2}
\sup_{y\in I_\delta} \;\;\; \babs{f_\delta\at{y}} + \babs{f_\delta^\prime\at{y}}\;\leq \; C\exp\at{-\frac{c}{\delta}}\,,
\qquad \int\limits_{I_\delta}\babs{ f_\delta\at{y}}\dint{y}\leq C\exp\at{-\frac{c}{\delta}}\,,
\end{align}
where we additionally used that
\begin{align*}
-\tfrac32\leq b_\delta y -1\leq-\tfrac12\,,\qquad +\tfrac12 \leq b_\delta y+1\leq+\tfrac32\qquad \text{for all}\quad y\in I_\delta.
\end{align*}
In the same way we demonstrate that
\begin{align}
\label{Prop:ConvTranScal.PEqn3}
\sup_{y\in I_\delta} \;\;\; \babs{g_\delta\at{y}} + \babs{g_\delta^\prime\at{y}}+
\babs{g_\delta^{\prime\prime}\at{y}}+
\babs{h_\delta\at{y}} + \babs{h_\delta^\prime\at{y}}+
\babs{h_\delta^{\prime\prime}\at{y}}\; \leq \; C\exp\at{-\frac{c}{\delta}}\,.
\end{align}
\ul{\emph{Estimates for $W_\delta$}}: \BMHC Thanks to the definition and evenness \EMHC of $S_\delta$, see \eqref{Eqn:DefTipScaling}, we derive from \eqref{Prop:ConvTranScal.PEqn0} the identities
\begin{align}
\label{Prop:ConvTranScal.PEqn4}
W_\delta\at{y}=\tfrac12 S_\delta^{\prime}\bat{ y}+\tfrac12 S_\delta^\prime\at{y_\delta^*}+ f_\delta\at{y}\,,\qquad
W_\delta^\prime\at{y}=\tfrac12 S_\delta^{\prime\prime}\bat{ y}+ f_\delta^\prime\at{y}\,,
\end{align}
so \eqref{Prop:ConvTranScal.Eqn1}$_1$ follows from
\eqref{Eqn:LimW}, \eqref{Prop:ConvTranScal.PEqn2}, and Proposition \ref{Prop:ConvTipScal} because the properties of $S_0$ and \eqref{Eqn:UpperBoundPrmB} ensure that
\begin{align*}
\babs{ S_0^\prime\at{y_\delta^*}-2}\leq C \exp\at{-cy_\delta^*}\leq C\delta\,.
\end{align*}
Moreover, by \eqref{Prop:ConvTipScal.Eqn2}+\eqref{Prop:ConvTranScal.PEqn2}+\eqref{Prop:ConvTranScal.PEqn4} we find \eqref{Prop:ConvTranScal.Eqn1}$_2$ via
\begin{align*}
\int\limits_{I_\delta}\babs{W_\delta\at{y}-W_0\at{y}}\dint{y}\leq C\int\limits_{I_\delta}\babs{S_\delta^\prime\at{y}-S_0^\prime\at{y}}\dint{y}+C\abs{
y_\delta^*S_\delta^\prime\at{y_\delta^*}-y_\delta^*S_0^\prime\at{y_\delta^*}}+C\delta\;\leq \;C\delta\,.
\end{align*}
\ul{\emph{Estimates for $T_\delta$}}:  According to \eqref{Prop:ConvTranScal.PEqn1}, we have
\begin{align*}
T_\delta\at{y}&=
-\frac{1}{2\delta}\int\limits_{-1}^{b_\delta y-\tfrac12}
\Bat{R_\delta^{\prime}\bat{x+\tfrac12}-
R_\delta^{\prime}\bat{-\tfrac12}}
\dint{x}+g_\delta\at{y}+h_\delta\at{y}
\\&=-\frac{1}{2\delta}R_\delta\bat{b_\delta y}+
\frac{1}{2\delta}R_\delta\bat{\tfrac12}-\frac{1}{2\delta}R_\delta^\prime\at{\tfrac12}\Bat{b_\delta {y}+\tfrac12}+g_\delta\at{y}+h_\delta\at{y}
\end{align*}
and using that \eqref{Eqn:DefTipScaling} implies
\begin{align*}
-R_\delta\at{b_\delta y}+R_\delta\at{\tfrac12}=\delta\Bat{S_\delta\at{y}-S_\delta\at{y_\delta^*}}\,,\qquad R_\delta^\prime\at{\tfrac12}=-\frac{\delta}{b_\delta} S_\delta^\prime\at{y_\delta^*}
\end{align*}
we arrive at
\begin{align}
\label{Prop:ConvTranScal.PEqn5}
T_\delta\at{y}&=\tfrac12S_\delta\at{y}+\tfrac12 S_\delta^\prime\at{y_\delta^*}y+\tfrac12\Bat{
y_\delta^*S_\delta^\prime\at{y_\delta^*}-
S_\delta\at{y_\delta^*}
}+g_\delta\at{y}+h_\delta\at{y}
\end{align}
and
\begin{align*}
T_\delta^\prime\at{y}=\tfrac12S_\delta^\prime\at{y}+\tfrac12 S_\delta^\prime\at{y_\delta^*}+g_\delta^\prime\at{y}+h_\delta^\prime\at{y}\,,\qquad
T_\delta^{\prime\prime}\at{y}=\tfrac12S_\delta^{\prime\prime}\at{y}+g_\delta^{\prime\prime}\at{y}+h_\delta^{\prime\prime}\at{y}\,.
\end{align*}
Moreover, thanks to the explicit formula \eqref{Eqn:limS} as well as \eqref{Eqn:UpperBoundPrmB} we readily justify 
\begin{align*}
\babs{y_\delta^*S_0^\prime\at{y_\delta^*}-S_0\at{y_\delta^*}-2\ln{2}}\leq C\delta\,,
\end{align*}
so \eqref{Prop:ConvTranScal.Eqn2} is a consequence
of \eqref{Prop:ConvTipScal.Eqn1}, \eqref{Eqn:limT}, and \eqref{Prop:ConvTranScal.PEqn3}.
\end{proof}
%
%
%
\subsection{Scaling of the scalar quantities}\label{sect:ScalLaws}
%
The following result exploits the approximation under the tip and the transition scaling in order to characterize the asymptotic behavior of the amplitude parameter $a_\delta$, the width parameter $b_\delta$, and the speed parameter $\la_\delta$ from \eqref{Eqn:TWInt} and \eqref{Eqn:DefPrmsAandB}.
\begin{proposition}[scaling relations for $a_\delta$, $b_\delta$, and $\la_\delta$]
\label{Prop:ScalingRelations}
We have
\begin{align*}
b_\delta = 2\delta-2\delta^2+\DO{\delta^3}\,,\qquad a_\delta=\delta+\at{2\ln2-1}\delta^2+\DO{\delta^3}
\end{align*}
as well as
\begin{align*}
\la_\delta^2=\frac{1}{\delta^\mu}\exp\at{1+\frac{1}{\delta}}\Bat{1 +\DO{\delta}}
\end{align*}
for all sufficiently small $\delta>0$.
\end{proposition}
\begin{proof}
The unimodality of $V_\delta$ combined \BMHC with \EMHC 
\eqref{Eqn:LowerBoundPrmB}+\eqref{Eqn:DefTranScaling}+\eqref{Prop:ConvTranScal.PEqn2}+\eqref{Prop:ConvTranScal.PEqn4} provides
\begin{align}
\label{Prop:ScalingRelations.Eqn1}
0\leq \sup_{1\leq \abs{x}\leq L}{V_\delta\at{x}}\leq V_\delta\at{-1}\leq \frac{\delta}{b_\delta}W_\delta\at{-y_\delta^*}\leq \frac{\delta}{b_\delta} f_\delta\at{-y_\delta^*} \leq C\exp\at{-\frac{c}{\delta}}=\DO{\delta^2}\,.
\end{align}
The norm constraint $\norm{V_\delta}_2=1$, see Proposition \ref{Prop:Existence},  therefore yields
\begin{align*}
1=2\int\limits_{-1}^0 V_\delta\at{x}^2\dint{x}+\DO{\delta^2}=\frac{2\delta^2}{b_\delta}\int\limits_{I_\delta}W_\delta\at{y}^2\dint{y}+\DO{\delta^2}\,,
\end{align*}
and since $S_\delta$ is even we compute
\begin{align*}
4\int\limits_{I_\delta}W_\delta\at{y}^2\dint{y}&=\int\limits_{I_\delta}\at{S_\delta^\prime\at{y}+
S_\delta^\prime\at{y_\delta^*}}^2\dint{y}+\DO{\delta}\\&
=
\int\limits_{I_\delta}\at{S_\delta^\prime\at{y}}^2\dint{y}+\BMHC \frac{S_\delta^\prime\at{y_\delta^*}^2}{b_\delta}\EMHC
+\DO{\delta}
\\&=
\int\limits_{I_\delta}\at{S_0^\prime\at{y}}^2\dint{y}+\BMHC\frac{S_0^\prime\at{y_\delta^*}^2}{b_\delta}\EMHC +\DO{\delta}\,,
\end{align*}
where we additionally used
\begin{align*}
\abs{\int\limits_{I_\delta}\at{S_\delta^\prime\at{y}}^2\dint{y}-\int\limits_{I_\delta}\at{S_0^\prime\at{y}}^2\dint{y}}\leq C \int\limits_{I_\delta} \babs{S_\delta^\prime\at{y}-S_0^\prime\at{y}} \dint{y}
\end{align*}
as well as \eqref{Prop:ConvTipScal.Eqn1}+\eqref{Prop:ConvTipScal.Eqn2}+\eqref{Eqn:LowerBoundPrmB}+\eqref{Prop:ConvTranScal.PEqn2}+\eqref{Prop:ConvTranScal.PEqn4}. By direct computations we justify that
\begin{align*}
\int\limits_{I_\delta}\at{S_0^\prime\at{y}}^2\dint{y}=\frac{4}{b_\delta}-8+\DO{\delta}\,,
\qquad \frac{S_0^\prime\at{y_\delta^*}^2}{b_\delta}
=\frac{4}{b_\delta}+\DO{\delta}
\end{align*}
holds due to \eqref{Eqn:UpperBoundPrmB} and the exponential decay of $S_0^\prime$. In summary, we obtain
\begin{align*}
1=4\frac{\delta^2}{b_\delta^2}\Bat{1-b_\delta+\DO{\delta b_\delta}}+\DO{\delta^2}
\end{align*}
and hence the desired expansion for $b_\delta$. Evaluating \BMHC both \eqref{Eqn:DefTipScaling} and \eqref{Eqn:DefFootScaling} for $y=y^*_\delta$ \EMHC gives
\begin{align*}
\frac{\delta}{a_\delta}-\delta S_\delta\at{y_\delta^*}= R_\delta\at{\tfrac12}=\delta T_\delta\at{y_\delta^*}\,,
\end{align*}
while $T_\delta\at{y_\delta^*}=y_\delta^* S_\delta^\prime\at{y_\delta^*}+\DO{\delta}$ is a direct consequence of \eqref{Prop:ConvTranScal.PEqn3}+\eqref{Prop:ConvTranScal.PEqn5}. Rearranging terms we find
\begin{align*}
\frac{\delta}{a_\delta}=2\delta y_\delta^* S_0^\prime\at{y_\delta^*}+\delta\Bat{ S_0\at{y_\delta^*}-y_\delta^*S_0^\prime\at{y_\delta^*}}+\DO{\delta^2}
\end{align*}
thanks to Proposition \ref{Prop:ConvTipScal}, so the properties of $S_0$ provide via
\begin{align*}
\frac{\delta}{a_\delta}=2\frac{\delta}{b_\delta}-\delta \ln4+\DO{\delta^2}
\end{align*}
the expansion of $a_\delta$. Finally, by \eqref{Eqn:DefPrmsAandB} we 
obtain
\begin{align*}
\delta^\mu\la_\delta^2\exp\at{-\frac{1}{\delta}}=\frac{b_\delta^2}{\delta^2}\frac{\delta^\mu}{a_\delta^\mu}\exp\at{\frac{1}{a_\delta}-\frac{1}{\delta}}\,,
\end{align*}
and this implies the law for $\la_\delta$.
\end{proof}
As already discussed at the end of \S\ref{sect:TipScal}, our asymptotic estimates can be improved for the Toda potential \eqref{Eqn:TodaPot}. Specifically, the error bounds in \eqref{Prop:ConvTranScal.Eqn1} and \eqref{Prop:ConvTranScal.Eqn2} are exponentially small in $\delta$ and provide more accurate asymptotic expansions in Proposition \ref{Prop:ScalingRelations}.
%
%
\subsection{Approximation formulas}\label{sect:Formulas}
%
%
We are now able to proof the main result from \S\ref{sec:Intro} about the approximation formulas  \eqref{Eqn:AppProfiles}. Numerical results concerning the approximation error are shown in Figure~\ref{Fig:AppErrors}.
\begin{figure}[t!] %
\centering{ %
\includegraphics[width=0.95\textwidth]{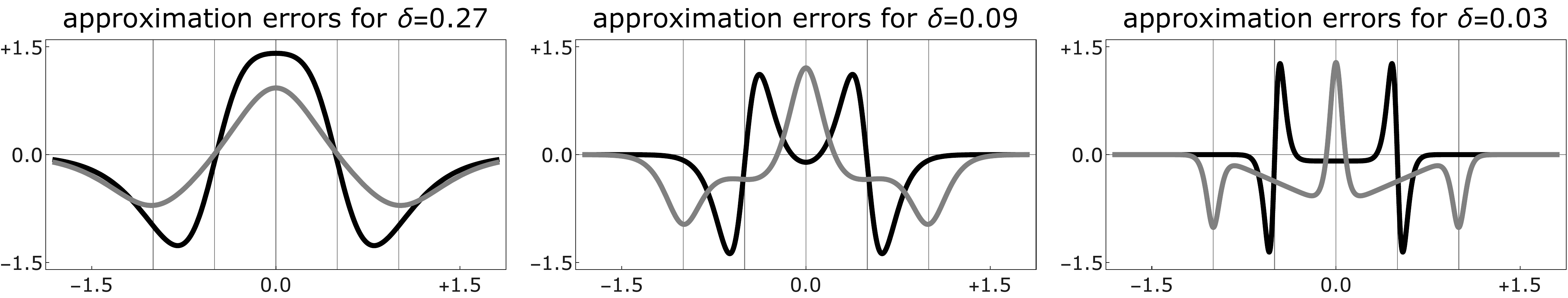} %
} %
\caption{Scaled approximation errors $c_1\delta^{-1}\at{V_\delta-\ol{V}_\delta}$ (black)
and $c_2\delta^{-2}\at{R_\delta-\ol{R}_\delta}$ (gray) as function of $x$ for the data from Figure \ref{Fig:Profiles}, where the artificial prefactors 
$c_1=0.5$ and $c_2=2.0$ are chosen for better illustration. The main contributions to the approximation error for the velocity
profiles stem from the vicinity of $x=\pm\tfrac12$, so $\norm{V_\delta-\ol{V}_\delta}_1$ 
is in fact asymptotically smaller than $\norm{V_\delta-\ol{V}_\delta}_\infty$.  
} %
\label{Fig:AppErrors}
\end{figure} %
\begin{theorem}[approximation formulas]
\label{Thm:Approximation}
There exists a constant $C$ such that
\begin{align*}
\norm{V_\delta-\ol{V}_\delta}_p\leq C\delta^{1+1/p}\,,\qquad \norm{R_\delta-\ol{R}_\delta}_p\leq C\delta^2
\end{align*}
\BMHC hold \EMHC for any $1\leq p \leq \infty$ and all sufficiently small $\delta>0$.
\end{theorem}
\begin{proof} We prove the assertions for $p=1$ and $p=\infty$ only; the 
estimates in the general case follow by interpolation, i.e., from $\norm{\cdot}_p\leq
\norm{\cdot }_1^{1/p}\norm{\cdot }_\infty^{1-1/p}$.
\par
\ul{\emph{Approximation of $V_\delta$}}: Since both $V_\delta$ and $\ol{V}_\delta$ are even functions and because 
\begin{align}
\label{Thm:Approximation.PEqn1}
\sup_{1\leq \abs{x}\leq L}\babs{V_\delta\at{x}}+\abs{\ol{V}_\delta\at{x}}\leq C\delta^2
\end{align}
is guaranteed by \eqref{Eqn:AppProfiles} and \eqref{Prop:ScalingRelations.Eqn1},
it is sufficient to establish the desired estimates on the domain
\begin{align*}
-1\leq x\leq 0\,.
\end{align*} 
Moreover, the function $\widehat{V}_\delta$ with
\begin{align*}
\widehat{V}_\delta\at{x}:=\frac{\delta}{b_\delta}W_0\at{\frac{x+\tfrac12}{b_\delta}}=\frac{\delta}{b_\delta}\at{\tanh\at{\frac{x+\tfrac12}{b_\delta}}+1}
\end{align*}
satisfies
\begin{align}
\label{Thm:Approximation.PEqn2}
\sup_{-1\leq x\leq0} \babs{\widehat{V}_\delta\at{x}-V_\delta\at{x}}\leq 
\sup_{y\in{I_\delta}} \frac{\delta}{b_\delta}\babs{W_0\at{y}-W_\delta\at{y}}\leq C\delta
\end{align}
and
\begin{align}
\label{Thm:Approximation.PEqn3}
\int\limits_{-1}^{0} \babs{\widehat{V}_\delta\at{x}-V_\delta\at{x}}\dint{x}\leq 
\delta\int\limits_{I_\delta}\babs{W_0\at{y}-W_\delta\at{y}}\dint{y}\leq C\delta^2
\end{align}
according to \eqref{Eqn:DefTranScaling}, Proposition \ref{Prop:ConvTranScal}, and Proposition \ref{Prop:ScalingRelations}. In order to bound the difference between $\widehat{V}_\delta$ and $\ol{V}_\delta$, we write
\begin{align*}
\ol{V}_\delta\at{x}=
\frac{1+\delta}{2}\at{\tanh\at{\frac{x+\tfrac12}{2\delta}}-\tanh\at{\frac{x-\tfrac12}{2\delta}}}
\end{align*}
and observe that 
\begin{align*}
\abs{\frac{1}{b_\delta}-\frac{1}{2\delta}}\leq C
\,,\qquad
\abs{\frac{\delta}{b_\delta}-\frac{1+\delta}{2}}\leq C\delta^2
\,,\qquad
\sup\limits_{-1\leq x\leq0}\abs{\tanh\at{\frac{x-\tfrac12}{2\delta}}+1}\leq C\delta^2
\end{align*}
as well as \BMHC
\begin{align*}
\abs{\tanh\at{\frac{x+\tfrac12}{2\delta}}-\tanh\at{\frac{x+\tfrac12}{b_\delta}}}\leq
C\at{
\sup_{\xi\in\Rset}\abs{\xi\tanh^\prime\at{\xi}}}\abs{\frac{b_\delta-2\delta}{\delta}}\leq C\delta
\end{align*}
and 
\begin{align*}
\int\limits_{-1}^0\abs{\tanh\at{\frac{x+\tfrac12}{2\delta}}-\tanh\at{\frac{x+\tfrac12}{b_\delta}}}\dint{x}&=
b_\delta\int\limits_{-1/b_\delta}^{+1/b_\delta}\abs{\tanh\at{\frac{b_\delta}{2\delta}y}-\tanh\at{y}}\dint{y}
\leq C\delta^2
\end{align*}
are \EMHC granted by Proposition \ref{Prop:ScalingRelations} and the properties of the hyperbolic tangent. We therefore \BMHC find the desired \EMHC estimates for $V_\delta-\ol{V}_\delta$ \BMHC by combining \eqref{Thm:Approximation.PEqn2} and \eqref{Thm:Approximation.PEqn3} with \EMHC
\begin{align*}
\sup_{-1\leq x\leq0} \babs{\widehat{V}_\delta\at{x}-\ol{V}_\delta\at{x}}\leq C\delta\,,\qquad
\int\limits_{-1}^0\babs{\widehat{V}_\delta\at{x}-\ol{V}_\delta\at{x}}\dint{x}\leq C\delta^2\,.
\end{align*}
\ul{\emph{Approximation of $R_\delta$}}: Direct computations yield
\begin{align*}
\ol{V}_\delta\at{x}\quad\xrightarrow{x\to\pm\infty}\quad0
\,,\qquad
\ol{R}_\delta\at{x}\quad\xrightarrow{x\to\pm\infty}\quad0
\end{align*}
as well as 
\begin{align*}
\ol{R}_\delta^{\,\prime}\at{x}=\ol{V}_\delta\at{x+\tfrac12}-\ol{V}_\delta\at{x-\tfrac12}\qquad \text{for all}\quad x\in\Rset\,,
\end{align*}
and using the definition of $\calA$ in \eqref{Eqn:DefOpA} we show that
\begin{align*}
\ol{R}_\delta=\calA \ol{V}_\delta
\end{align*}
holds without any approximation error. Combining this with $R_\delta=\calA V_\delta$ and \BMHC Young's inequality for convolutions \EMHC we verify
\begin{align*}
\bnorm{\ol{R}_\delta-R_\delta}_1\leq
\bnorm{\ol{V}_\delta-V_\delta}_1\,,\qquad 
\bnorm{\ol{R}_\delta-R_\delta}_\infty\leq
\bnorm{\ol{V}_\delta-V_\delta}_1\,,
\end{align*}
and the proof is complete.
\end{proof}
The definitions of $\ol{V}_\delta$ and $\ol{R}_\delta$ in \eqref{Eqn:AppProfiles} imply
\begin{align*}
\norm{V_0-\ol{V}_\delta}_p\sim\delta^{1/p}\,,\qquad \norm{R_0-\ol{R}_\delta}_p\sim{\delta}\,,
\end{align*}
\BMHC so \EMHC $\ol{V}_\delta$ and $\ol{R}_\delta$ can in fact be regarded 
as the leading-order terms in the high-energy limit and satisfy \eqref{Eqn:AppResult}. We also mention that the bound for $\norm{V_\delta-\ol{V}_\delta}_1$ from Theorem \ref{Thm:Approximation} depends on $L$ due to the rather rough estimate of $V_\delta\at{x}-\widehat{V}_\delta\at{x}$ on the domain $1\leq\abs{x}\leq L$, see \eqref{Thm:Approximation.PEqn1}. However, adapting the arguments from \cite{HR10,HM15} we can establish exponential decay estimates for $V_\delta$ and in combination with the tail behavior of $\ol{V}_\delta$ we obtain improved bounds which are independent of $L$ and allow us to pass to the solitary limit $L\to\infty$. Our main result therefore covers solitary waves, too.
\par
We finally mention that the error bounds in Theorem \ref{Thm:Approximation} 
are optimal for the Toda chain \eqref{Eqn:TodaPot} because we still
have 
\begin{align*}
\frac{1}{b_\delta}-\frac{1}{2\delta}\quad\xrightarrow{\delta\to0}\quad \tfrac12
\end{align*} 
and hence a rather large difference between $\widehat{V}_\delta$ and $\ol{V}_\delta$. The approximation error between $\widehat{V}_\delta$ and $V_\delta$, however, is exponentially small in $\delta$.
%
%
%
\appendix
\section{Proof of Proposition \ref{Prop:Existence}}\label{app}
%
The key arguments have been developed in \cite{Her10} and can be sketched as follows.
\par
\ul{\emph{Existence of maximizer}}: By standard arguments we show that the convex cone $\calC$ from \eqref{Eqn:DefCone} is closed under weak $\fspaceL^2$-convergence. Moreover, the convolution operator $\calA$ from \eqref{Eqn:DefOpA} satisfies
\begin{align}
\label{Eqn:App2}
\norm{\calA V}_\infty\leq \norm{V}_2
\end{align}
and is -- in the periodic setting $L<\infty$ -- compact. These observations imply that
\BMHC the energy functional $\calP_\delta$ from \eqref{Eqn:DefOpA} is continuous with respect to the \EMHC weak $\fspaceL^2$-topology and attains its maximum on the weakly compact set
\begin{align}
\label{Eqn:App1}
\ol{B}_1\at{0}\cap \calC\,,
\end{align}
where $\ol{B}_1\at{0}$ denotes the closed unit ball with respect to $\norm{\cdot}_2$. Furthermore, since the map
\begin{align*}
s\quad \mapsto \quad \calP_\delta\at{s V}
\end{align*}
is -- due to the monotonicity of $\Phi^\prime$ -- for any fixed $V\in\calC$ strictly increasing with respect to $s\geq 0$, we conclude that each maximizer $V_\delta$ of $\calP_\delta$ in the set \eqref{Eqn:App1} belongs to the sphere $\partial\ol{B}_1\at{0}$. 
\par
\ul{\emph{Verification of the traveling wave equation}}:
Due to the shape constraint $V_\delta\in\calC$ it is not obvious that 
\eqref{Eqn:TWInt} is the Euler-Lagrange equation for a maximizer $V_\delta$ but we can argue as follows: We introduce an `improvement operator' $\calF_\delta$ by
\begin{align*}
\calF_\delta\at{V}:= \frac{\calG_\delta\at{V}}{\norm{\calG_\delta\at{V}}_2}\,,\qquad \calG_\delta\at{V}:=\partial \calP_\delta\at{V}=\frac{1}{\delta}\calA \Phi^\prime\at{\frac{\calA V_\delta}{\delta}}\,,
\end{align*}
which is well defined on the set
\begin{align}
\label{Eqn:App3}
\partial \ol{B}_1\at{0}\cap \calC
\end{align}
due to $\calG_\delta\at{V}\neq0$. By direct computation we verify that $\calC$ is invariant under both the action of $\calA$ and the superposition operator corresponding to $\Phi^\prime$, and this implies that \eqref{Eqn:App3} is invariant under the action of $\calF_\delta$. Moreover, thanks to the convexity \BMHC inequality for $\calP_\delta$ \EMHC and the identity $\norm{\calF_\delta\at{V}}_2=\norm{V}_2=1$ we obtain
\begin{align}
\label{Eqn:App4}
\begin{split}
\calP_\delta\bat{\calF_\delta\at{V}}-\calP_\delta\bat{V}&\geq \bskp{\calG_\delta\at{V}}{\calF_\delta\at{V}-V}
\\&=
\norm{\calG_\delta\at{V}}_2\bskp{\calF_\delta\at{V}}{\calF_\delta\at{V}-V}
\\&=
\tfrac12\norm{\calG_\delta\at{V}}_2\norm{\calF_\delta\at{V}-V}_2^2\geq0\,,
\end{split}
\end{align}
where $\skp{\cdot}{\cdot}$ stands for the usual $\fspace{L}^2$-inner product of $2L$-periodic functions, so applying $\calF_\delta$ increases the potential energy. In particular, any maximizer $V_\delta$ satisfies $\calP_\delta\bat{V_\delta}=\calP_\delta\bat{\calF_\delta \at{V_\delta}}$ and \eqref{Eqn:App4} implies via $\norm{\calF_\delta\at{V_\delta}-V_\delta}_2=0$ the fixed-point identity
\begin{align*}
V_\delta=\calF_\delta\at{V_\delta}\,,
\end{align*}
which is equivalent to \eqref{Eqn:TWInt} with $\la_\delta=\sqrt{\norm{\calG_\delta\at{V_\delta}}_2}>0$. In other words, the shape constraint does not contribute to the Euler-Lagrange equation and $\la_\delta^2$ can be viewed as the Lagrangian multiplier that stems from the norm constraint. A similar result has been derived in \cite{KS12,KS13} by different arguments.
\par
\ul{\emph{Convergence in the high-energy limit}}: 
\BMHC Thanks to $\Phi\at{r}\sim r^\mu \exp\at{r}$ for $r\gg1$ \EMHC we observe that
\begin{align*}
\calP_\delta\at{V_\delta}\geq \calP_\delta\at{V_0}=\int\limits_{-L}^{+L}\Phi\at{\frac{R_0\at{x}}{\delta}}\dint{x}=
2\int\limits_0^1\Phi\at{\frac{r}{\delta}}\dint{r}\geq c\exp\at{\frac{1}{\delta}}\at{\frac{1}{\delta}}^{\mu-1}\,,
\end{align*}
\BMHC where $V_0$ and $R_0=\calA V_0$ are defined in Proposition \ref{Prop:Existence}. \EMHC
On the other hand, \BMHC with $R_\delta=\calA V_\delta$ \EMHC we estimate
\begin{align*}
\calP_\delta\at{V_\delta}\leq 2L \Phi\at{\frac{R_\delta\at{0}}{\delta}} \leq C\exp\at{\frac{R_\delta\at{0}}{\delta}}\at{\frac{R_\delta\at{0}}{\delta}}^\mu\,,
\end{align*}
and combining this with 
$0\leq R_\delta\at{0}\leq \norm{V_\delta}_2=1$
we arrive at
\begin{align*}
R_\delta\at{0}\quad\xrightarrow{\delta\to0}\quad1
\end{align*}
because otherwise the lower bound for $\calP_\delta\at{V_\delta}$ would asymptotically exceed the upper bound. The convergence result \eqref{Prop:Existence.Eqn1} is now a direct consequence of
\begin{align*}
\bnorm{V_\delta-V_0}_2^2=\bnorm{V_\delta}_2^2+\bnorm{V_0}_2^2-2\bskp{V_\delta}{V_0}=
2-2R_\delta\at{0}
\end{align*}
and the estimate \eqref{Eqn:App2}.
\par
\ul{\emph{Remarks}}:
An elementary discretization (Riemann sums instead of integrals) of the improvement dynamics
\begin{align*}
V\quad \rightsquigarrow\quad \calF_\delta\at{V}\quad \rightsquigarrow\quad
\calF_\delta\bat{\calF_\delta\at{V}}\quad \rightsquigarrow\quad ...
\end{align*}
was used to compute the numerical wave profiles from Figure \ref{Fig:Profiles}, see \cite{FV99,EP05} for similar schemes. We finally mention that Proposition \ref{Prop:Existence} holds also in the solitary case $L=\infty$. The proof, however, is more complicate because the operator $\calA$ is not compact anymore. The strong compactness of any maximizing sequence has therefore to be derived from the superquadratic grow of $\Phi$ and a variant of the concentration compactness principle, see again \cite{Her10} for the details.
\section*{Acknowledgement}
The author gratefully acknowledges the support by the \emph{Deutsche Forschungsgemeinschaft} (DFG individual grant HE 6853/2-1).
%

%
%

\begin{thebibliography}{CBCCMK14}

\bibitem[BP13]{BP13}
M.~Betti and D.~E. Pelinovsky.
\newblock Periodic traveling waves in diatomic granular chains.
\newblock {\em J. Nonlinear Sci.}, 23(5):689--730, 2013.

\bibitem[CBCCMK14]{CCCK14}
M.~Chirilus-Bruckner, C.~Chong, J.~Cuevas-Maraver, and P.~G. Kevrekidis.
\newblock Sine-{G}ordon equation: from discrete to continuum.
\newblock In {\em The sine-{G}ordon model and its applications}, volume~10 of
  {\em Nonlinear Syst. Complex.}, pages 31--57. Springer, Cham, 2014.

\bibitem[DHM06]{DHM06}
W.~Dreyer, M.~Herrmann, and A.~Mielke.
\newblock Micro-macro transition for the atomic chain via {Whitham}'s
  modulation equation.
\newblock {\em Nonlinearity}, 19(2):471--500, 2006.

\bibitem[DP14]{DP14}
E.~Dumas and D.~E. Pelinovsky.
\newblock Justification of the {log-KdV} equation in granular chains: the case
  of precompression.
\newblock {\em SIAM J. Math. Anal.}, 46(6):4075--4103, 2014.

\bibitem[EP05]{EP05}
J.~M. English and R.~L. Pego.
\newblock On the solitary wave pulse in a chain of beads.
\newblock {\em Proc. Amer. Math. Soc.}, 133(6):1763--1768 (electronic), 2005.

\bibitem[FDMF03]{FDMF02}
S.~Flach, J.~Dorignac, A.E. Miroshnichenko, and V.~Fleurov.
\newblock Discrete breathers close to the anticontinuum limit: existence and
  wave scattering.
\newblock In {\em Nonlinear physics: theory and experiment, {II} ({G}allipoli,
  2002)}, pages 57--63. World Sci. Publ., River Edge, NJ, 2003.

\bibitem[FM02]{FM02}
G.~Friesecke and K.~Matthies.
\newblock Atomic-scale localization of high-energy solitary waves on lattices.
\newblock {\em Phys. D}, 171(4):211--220, 2002.

\bibitem[FM03]{FM03}
G.~Friesecke and K.~Matthies.
\newblock Geometric solitary waves in a 2d mass spring lattice.
\newblock {\em Discrete Contin. Dyn. Syst. Ser. B}, 3:105--114, 2003.

\bibitem[FML14]{FM14}
G.~Friesecke and A.~Mikikits-Leitner.
\newblock Cnoidal waves on {F}ermi-{P}asta-{U}lam lattices.
\newblock to appear in J. Dyn. Diff. Equat., available via Springer online
  first, 2014.

\bibitem[FP99]{FP99}
G.~Friesecke and R.~L. Pego.
\newblock Solitary waves on {FPU} lattices. {I}. {Q}ualitative properties,
  renormalization and continuum limit.
\newblock {\em Nonlinearity}, 12(6):1601--1627, 1999.

\bibitem[FP02]{FP02}
G.~Friesecke and R.~L. Pego.
\newblock Solitary waves on {FPU} lattices. {II}. {L}inear implies nonlinear
  stability.
\newblock {\em Nonlinearity}, 15(4):1343--1359, 2002.

\bibitem[FP04a]{FP04a}
G.~Friesecke and R.~L. Pego.
\newblock Solitary waves on {F}ermi-{P}asta-{U}lam lattices. {III}.
  {H}owland-type {F}loquet theory.
\newblock {\em Nonlinearity}, 17(1):207--227, 2004.

\bibitem[FP04b]{FP04b}
G.~Friesecke and R.~L. Pego.
\newblock Solitary waves on {F}ermi-{P}asta-{U}lam lattices. {IV}. {P}roof of
  stability at low energy.
\newblock {\em Nonlinearity}, 17(1):229--251, 2004.

\bibitem[FV99]{FV99}
A.-M. Filip and S.~Venakides.
\newblock Existence and modulation of traveling waves in particle chains.
\newblock {\em Comm. Pure Appl. Math.}, 51(6):693--735, 1999.

\bibitem[FW94]{FW94}
G.~Friesecke and J.~A.~D. Wattis.
\newblock Existence theorem for solitary waves on lattices.
\newblock {\em Comm. Math. Phys.}, 161(2):391--418, 1994.

\bibitem[GMWZ14]{GMWZ14}
J.~Gaison, S.~Moskow, J.~D. Wright, and Q.~Zhang.
\newblock Approximation of polyatomic {FPU} lattices by {K}d{V} equations.
\newblock {\em Multiscale Model. Simul.}, 12(3):953--995, 2014.

\bibitem[Her10]{Her10}
M.~Herrmann.
\newblock Unimodal wavetrains and solitons in convex {F}ermi-{P}asta-{U}lam
  chains.
\newblock {\em Proc. Roy. Soc. Edinburgh Sect. A}, 140(4):753--785, 2010.

\bibitem[HM15]{HM15}
M.~Herrmann and K.~Matthies.
\newblock Asymptotic formulas for solitary waves in the high-energy limit of
  {FPU}-type chains.
\newblock {\em Nonlinearity}, 28(8):2767--2789, 2015.

\bibitem[HML15]{HML15}
M.~Herrmann and A.~Mikikits-Leitner.
\newblock Kdv waves in atomic chains with nonlocal interactions.
\newblock arXiv:1503.02307, 2015.

\bibitem[HMSZ13]{HMSZ13}
M.~Herrmann, K.~Matthies, H.~Schwetlick, and J.~Zimmer.
\newblock Subsonic phase transition waves in bistable lattice models with small
  spinodal region.
\newblock {\em SIAM J. Math. Anal.}, 45(5):2625--2645, 2013.

\bibitem[HR10]{HR10}
M.~Herrmann and J.~D.~M. Rademacher.
\newblock Heteroclinic travelling waves in convex {FPU}-type chains.
\newblock {\em SIAM J. Math. Anal.}, 42(4):1483--1504, 2010.

\bibitem[HW08]{HW08}
A.~Hoffman and C.~E. Wayne.
\newblock Counter-propagating two-soliton solutions in the
  {F}ermi-{P}asta-{U}lam lattice.
\newblock {\em Nonlinearity}, 21(12):2911--2947, 2008.

\bibitem[HW09]{HW09}
A.~Hoffman and C.~E. Wayne.
\newblock Asymptotic two-soliton solutions in the {F}ermi-{P}asta-{U}lam model.
\newblock {\em J. Dynam. Differential Equations}, 21(2):343--351, 2009.

\bibitem[IJ05]{IJ05}
G.~Iooss and G.~James.
\newblock Localized waves in nonlinear oscillator chains.
\newblock {\em \mbox{Chaos}}, 15:015113, 2005.

\bibitem[Jam12]{Jam12}
G.~James.
\newblock Periodic travelling waves and compactons in granular chains.
\newblock {\em J. Nonlinear Sci.}, 22(5):813--848, 2012.

\bibitem[JP14]{GP14}
G.~James and D.~E. Pelinovsky.
\newblock Gaussian solitary waves and compactons in {F}ermi-{P}asta-{U}lam
  lattices with {H}ertzian potentials.
\newblock {\em Proc. R. Soc. Lond. Ser. A Math. Phys. Eng. Sci.},
  470(2165):20130462, 20, 2014.

\bibitem[MA94]{MKA94}
R.~S. MacKay and S.~Aubry.
\newblock Proof of existence of breathers for time-reversible or {H}amiltonian
  networks of weakly coupled oscillators.
\newblock {\em Nonlinearity}, 7(6):1623--1643, 1994.

\bibitem[Miz11]{Miz11}
T.~Mizumachi.
\newblock {$N$}-soliton states of the {F}ermi-{P}asta-{U}lam lattices.
\newblock {\em SIAM J. Math. Anal.}, 43:2170--2210, 2011.

\bibitem[Miz13]{Miz13}
T.~Mizumachi.
\newblock Asymptotic stability of {$N$}-solitary waves of the {FPU} lattices.
\newblock {\em Arch. Ration. Mech. Anal.}, 207(2):393--457, 2013.

\bibitem[Pan05]{Pan05}
A.~Pankov.
\newblock {\em Traveling {W}aves and {P}eriodic {O}scillations in
  {F}ermi-{P}asta-{U}lam {L}attices}.
\newblock Imperial College Press, London, 2005.

\bibitem[SK12]{KS12}
A.~Stefanov and P.~Kevrekidis.
\newblock On the existence of solitary traveling waves for generalized
  {H}ertzian chains.
\newblock {\em J. Nonlinear Sci.}, 22(3):327--349, 2012.

\bibitem[SK13]{KS13}
A.~Stefanov and P.~Kevrekidis.
\newblock Traveling waves for monomer chains with precompression.
\newblock {\em Nonlinearity}, 26(2):539--564, 2013.

\bibitem[SW97]{SW97}
D.~Smets and M.~Willem.
\newblock Solitary waves with prescribed speed on infinite lattices.
\newblock {\em J. Funct. Anal.}, 149:266--275, 1997.

\bibitem[SZ09]{SZ09}
H.~Schwetlick and J.~Zimmer.
\newblock Existence of dynamic phase transitions in a one-dimensional lattice
  model with piecewise quadratic interaction potential.
\newblock {\em SIAM J. Math. Anal.}, 41(3):1231--1271, 2009.

\bibitem[SZ12]{SZ12}
H.~Schwetlick and J.~Zimmer.
\newblock Kinetic relations for a lattice model of phase transitions.
\newblock {\em Arch. Rational Mech. Anal.}, 206:707--724, 2012.

\bibitem[Tes01]{Tes01}
G.~Teschl.
\newblock Almost everything you always wanted to know about the {T}oda
  equation.
\newblock {\em Jahresber. Deutsch. Math.-Verein.}, 103(4):149--162, 2001.

\bibitem[Tod67]{Tod67}
M.~Toda.
\newblock Vibrations in a chain with nonlinear interaction.
\newblock {\em J. Phys. Soc. Japan}, 22(2):431--436, 1967.

\bibitem[Tre04]{Tre04}
D.~Treschev.
\newblock Travelling waves in {FPU} lattices.
\newblock {\em Discrete Contin. Dyn. Syst.}, 11(4):867--880, 2004.

\bibitem[TV05]{TV05}
L.~Truskinovsky and A.~Vainchtein.
\newblock Kinetics of martensitic phase transitions: lattice model.
\newblock {\em SIAM J. Appl. Math.}, 66:533--553, 2005.

\bibitem[TV10]{TV10}
E.~Trofimov and A.~Vainchtein.
\newblock Shocks and kinks in a discrete model of displacive phase transitions.
\newblock to appear in Continuum Mechanics and Thermodynamics, 2010.

\bibitem[TV14]{TV14}
L.~Truskinovsky and A.~Vainchtein.
\newblock Solitary waves in a nonintegrable {F}ermi-{P}asta-{U}lam chain.
\newblock {\em Phys. Rev. E}, 90:042903:1--8, 2014.

\bibitem[Whi74]{Whi74}
G.B. Whitham.
\newblock {\em Linear and {Nonlinear} {Waves}}, volume 1237 of {\em Pure And
  Applied Mathematics}.
\newblock Wiley Interscience, New York, 1974.

\end{thebibliography}
\end{document}